\documentclass{amsart}

\usepackage{amssymb}
\usepackage{mathrsfs}
\usepackage{stmaryrd}
\usepackage{enumerate}
\usepackage{tikz}
\usepackage{aliascnt}
\usepackage[colorlinks=true, linkcolor=black, citecolor=magenta]{hyperref}

\newcommand{\Bibkeyhack}[3]{}

\newtheorem{theorem}{Theorem}

\newaliascnt{lemma}{theorem}
\newtheorem{lemma}[lemma]{Lemma}
\aliascntresetthe{lemma}

\newaliascnt{corollary}{theorem}
\newtheorem{corollary}[corollary]{Corollary}
\aliascntresetthe{corollary}

\newaliascnt{proposition}{theorem}
\newtheorem{proposition}[proposition]{Proposition}
\aliascntresetthe{proposition}

\theoremstyle{definition}

\newaliascnt{definition}{theorem}
\newtheorem{definition}[definition]{Definition}
\aliascntresetthe{definition}

\newaliascnt{remark}{theorem}
\newtheorem{remark}[remark]{Remark}
\aliascntresetthe{remark}

\newaliascnt{example}{theorem}
\newtheorem{example}[example]{Example}
\aliascntresetthe{example}

\newtheorem*{acknowledgements}{Acknowledgements}

\newcommand{\C}{\mathbb{C}}
\newcommand{\R}{\mathbb{R}}
\newcommand{\Q}{\mathbb{Q}}
\newcommand{\Z}{\mathbb{Z}}

\newcommand{\bG}{\mathbb{G}}
\newcommand{\bA}{\mathbb{A}}
\newcommand{\bL}{\mathbb{L}}
\newcommand{\bP}{\mathbb{P}}

\newcommand{\bOne}{\mathbf{1}}

\newcommand{\sM}{\mathscr{M}}

\newcommand{\sG}{\mathscr{G}}

\newcommand{\cA}{\mathcal{A}}

\newcommand{\cL}{\mathcal{L}}
\newcommand{\cG}{\mathcal{G}}
\newcommand{\cN}{\mathcal{N}}
\newcommand{\cS}{\mathcal{S}}
\newcommand{\cT}{\mathcal{T}}

\newcommand{\cF}{\mathcal{F}}

\DeclareMathOperator{\init}{in}
\DeclareMathOperator{\Trop}{Trop}
\DeclareMathOperator{\relint}{relint}
\DeclareMathOperator{\wt}{wt}
\DeclareMathOperator{\rk}{rk}
\DeclareMathOperator{\cl}{cl}
\DeclareMathOperator{\fact}{fact}

\newcommand{\barchi}{\overline{\chi}}

\newcommand{\topo}{\mathrm{top}}
\newcommand{\Var}{\mathbf{Var}}

\newcommand{\rat}{\mathrm{rat}}

\title{The motivic zeta functions of a matroid}

\author{David Jensen, Max Kutler, and Jeremy Usatine}

\address{David Jensen, Department of Mathematics, University of Kentucky}
\email{dave.jensen@uky.edu}

\address{Max Kutler, Department of Mathematics, University of Kentucky}
\email{max.kutler@uky.edu}

\address{Jeremy Usatine, Department of Mathematics, Brown University}
\email{jeremy{\_}usatine@brown.edu}

\begin{document}

\begin{abstract}
We introduce motivic zeta functions for matroids.  These zeta functions are defined as sums over the lattice points of Bergman fans, and in the realizable case, they coincide with the motivic Igusa zeta functions of hyperplane arrangements.  We show that these motivic zeta functions satisfy a functional equation arising from matroid Poincar\'{e} duality in the sense of Adiprasito--Huh--Katz. In the process, we obtain a formula for the Hilbert series of the cohomology ring of a matroid, in the sense of Feichtner--Yuzvinsky.  We then show that our motivic zeta functions specialize to the topological zeta functions for matroids introduced by van der Veer, and we compute the first two coefficients in the Taylor expansion of these topological zeta functions, providing affirmative answers to two questions posed by van der Veer.
\end{abstract}

\maketitle

\numberwithin{theorem}{section}
\numberwithin{lemma}{section}
\numberwithin{corollary}{section}
\numberwithin{proposition}{section}
\numberwithin{remark}{section}
\numberwithin{definition}{section}
\numberwithin{example}{section}

\section{Introduction}
\label{introduction}

In this paper, we introduce a notion of motivic zeta functions for matroids.  In the realizable case, these coincide with the motivic Igusa zeta functions of hyperplane arrangements.  Matroids, which generalize the combinatorics of hyperplane arrangements, have become a topic of renewed interest in algebraic geometry.  Recent developments have shown that many seemingly geometric properties of hyperplane arrangements are in fact not geometric at all, arising instead from the combinatorics of matroids \cite{FeichtnerYuzvinsky, AdiprasitoHuhKatz, EliasProudfootWakefield, Eur, LopezdeMedranoRinconShaw}.  This is striking because the vast majority of matroids cannot be realized by hyperplane arrangements \cite{Nelson}.

The motivic zeta functions for matroids satisfy many properties predicted by the geometry of motivic Igusa zeta functions.  For example, we prove that these zeta functions are always rational.  Even in the absence of a realizing hyperplane arrangement, we show that these zeta functions satisfy a functional equation coming from Poincar\'{e} duality for matroids in the sense of Adiprasito--Huh--Katz \cite{AdiprasitoHuhKatz}.  This functional equation implies an identity in terms of the characteristic polynomials of matroids.  We also prove that the motivic zeta function specializes to the topological zeta function of a matroid recently introduced by van der Veer \cite{vanderVeer}.  We show that motivic zeta functions satisfy a recurrence relation in terms of matroid minors, and we use this to compute the first two coefficients in the Taylor expansion of the topological zeta function, giving affirmative answers to two questions posed by van der Veer \cite{vanderVeer}.  In \autoref{arbitrarybuildingsets}, we demonstrate how our definition of motivic zeta functions allows us to prove that certain expressions in terms of building sets do not actually depend on the chosen building set.

Throughout the paper, we always assume that a matroid has a finite and non-empty ground set.  Everything in this paper applies to matroids regardless of realizability, but we often give geometric interpretations in the realizable case.  For ease of exposition, when we say that a matroid is realizable, we always mean realizable over $\C$.  Similarly, if we say that $\cA$ is a hyperplane arrangement, we will always mean that $\cA$ is a central essential hyperplane arrangement in some affine space over $\C$.

\subsection{The motivic zeta functions of a matroid}

We now introduce motivic zeta functions of matroids.  We refer the reader to \autoref{defWeight} and \autoref{defVert} for the relevant notation.

\begin{definition}
\label{definitionmotiviczetafunctions}
Let $M$ be a matroid with ground set $E$. We define the \textbf{motivic zeta function} of $M$ to be
\[
	Z_M(q, T) = \sum_{w \in \Z_{\geq 0}^E} \chi_{M_w}(q) q^{-\rk M - \wt_M(w)} T^{|w|} \in \Z[q^{\pm 1}]\llbracket T \rrbracket,
\]
we define the \textbf{local motivic zeta function} of $M$ to be
\[
	Z_{M}^0(q, T) = \sum_{w \in \Z_{> 0}^E} \chi_{M_w}(q) q^{-\rk M - \wt_M(w)} T^{|w|} \in \Z[q^{\pm 1}]\llbracket T \rrbracket,
\]
and we define the \textbf{reduced motivic zeta function} of $M$ to be
\[
	\overline{Z}_M(q, T) = \sum_{v \in \Z^E/\Z\bOne} \barchi_{M_v}(q) q^{-(\rk M-1) - \overline{\wt}_M(v)} T^{\overline{|v|}} \in \Z[q^{\pm 1}]\llbracket T \rrbracket.
\]
\end{definition}

\begin{remark}
See \autoref{relationshipbetweenvariousmotiviczetafunctions} below for a simple relationship between $Z_M(q,T)$, $Z_M^0(q,T)$, and $\overline{Z}_M(q,T)$.  Specifically, both $Z_M(q,T)$ and $Z_M^0(q,T)$ can be obtained from $\overline{Z}_M(q,T)$ by multiplying by explicit rational functions.  Because of this, \autoref{functionalequationforreducedmotiviczetafunction} and \autoref{recurrencerelationmotiviczetafunctionat0} below could be stated in terms of any of the three zeta functions, and our choices are based purely on elegance.
\end{remark}

\begin{remark}
The motivic zeta function of a matroid is a strictly more refined invariant than its characteristic polynomial.  To see this, we give an example in \autoref{sectionwithexamples} of a pair of matroids with the same characteristic polynomial but distinct motivic zeta functions.  We also give an example of a pair of non-isomorphic matroids with the same motivic zeta function.
\end{remark}

\begin{remark}
If a matroid contains a loop, then its characteristic polynomial is equal to zero.  Thus the sum defining $Z_M(q,T)$ (resp. $Z_{M}^0(q,T)$, $\overline{Z}_M(q,T)$) may be taken over those $w \in \Z_{\geq 0}^E$ (resp. $w\in \Z_{>0}^E$, $v \in \Z^E / \Z \bOne$) where $M_w$ (resp. $M_w$, $M_v$) is loopless.
\end{remark}

We briefly explain why it is appropriate to call these ``motivic'' zeta functions.  Let $K_0(\Var_\C)$ denote the Grothendieck ring of complex varieties, let $\bL \in K_0(\Var_\C)$ denote the class $[\bA_\C^1]$ of the affine line, and let $\sM_\C$ denote the ring obtained from $K_0(\Var_\C)$ by inverting $\bL$.  Kontsevich introduced ``motivic integration'' \cite{Kontsevich}, a theory of integration modeled on integration over $p$-adic manifolds. In this motivic integration, the $p$-adic manifold is replaced by the so-called arc scheme of a complex variety, and the integrals take values in the localized Grothendieck ring $\sM_\C$.  Denef and Loeser then used motivic integration to define the motivic Igusa zeta function of a hypersurface (sometimes also called the naive motivic zeta function of Denef and Loeser) \cite{DenefLoeser1998}.  This zeta function is a power series with coefficients in $\sM_\C$.  The motivic Igusa zeta function specializes to both Igusa's local zeta function and the topological zeta function of Denef and Loeser.  We refer to \cite{DenefLoeser2001} and \cite{CNS} for more information on the motivic Igusa zeta function and its applications.

If a matroid $M$ is realized by a hyperplane arrangement $\cA$, then $Z_M(\bL,T)$ and $Z_M^0(\bL,T)$ are equal to the motivic Igusa zeta function and the motivic Igusa zeta function at 0, respectively, of the arrangement $\cA$ \cite[Theorem 1.7]{KutlerUsatine}. Similarly if $\overline{\cA}$ is the projective arrangement obtained by projectivizing $\cA$, then $\overline{Z}_M(\bL,T)$ is equal to the motivic Igusa zeta function of $\overline{\cA}$ (see \autoref{rmk:realized}). Thus we call $Z_M(q,T)$, $Z_M^0(q,T)$, and $\overline{Z}_M(q,T)$ motivic zeta functions because, in the realizable case, they coincide with the motivic Igusa zeta functions of hyperplane arrangements.

\subsection{Rationality of the motivic zeta functions}

Whenever a new zeta function is introduced, it is reasonable to return to the Weil conjectures.  Specifically, one should ask whether the zeta function is rational, find a functional equation induced by Poincar\'{e} duality, and conduct a careful analysis of the function's zeros and poles.  In this paper, we study the first two of these questions.  In \autoref{rationalitysection}, we prove the following.

\begin{theorem}
\label{thm:rational}
Let $M$ be a matroid.  The motivic zeta functions $Z_M(q,T)$, $Z_M^0(q,T)$, and $\overline{Z}_M(q,T)$ are all rational functions in $q$ and $T$.
\end{theorem}

Indeed, \autoref{rationalformulaformotiviczetafunctionmaximalbuildingset} gives an explicit formula for each of these three functions as rational functions.  In the realizable case, this formula agrees with that obtained by applying \cite[Corollary 3.3.2]{DenefLoeser2001} to the log resolution given by a wonderful model of the corresponding hyperplane arrangement.  In particular, this corresponds to considering the wonderful model with respect to the ``maximal building set''.  There are other wonderful models corresponding to other ``building sets'' \cite{DeConciniProcesi}, and these other building sets have been defined for non-realizable matroids by Feichtner and Kozlov \cite{FeichtnerKozlov}.  In \autoref{arbitrarybuildingsets}, we generalize our results to other building sets.  In particular, \autoref{rationalformulaformotiviczetafunction} gives alternate formulas for $Z_M(q,T)$, $Z_M^0(q,T)$, and $\overline{Z}_M(q,T)$ as rational functions in terms of arbitrary building sets.  These formulas provide a straightforward proof that certain expressions do not depend on the choice of building set.  This illustrates another philosophical parallel with motivic integration, which is often used to prove that certain constructions in terms of log resolutions do not actually depend on the choice of log resolution.

\subsection{A functional equation}

The motivic Igusa zeta function of a projective hypersurface satisfies a certain functional equation \cite[Chapter 7 Proposition 3.3.10]{CNS}.  The functional equation is in terms of a duality morphism $\sM_\C \to \sM_\C$ introduced by Bittner \cite[Corollary 3.4]{Bittner},  which can be thought of as a motivic incarnation of Poincar\'{e} duality \cite[Chapter 2 Remark 5.1.9]{CNS}.  Our second main result shows that the reduced motivic zeta function of a matroid $M$ satisfies this functional equation even when $M$ is not realizable.

\begin{theorem}
\label{functionalequationforreducedmotiviczetafunction}
Let $M$ be a matroid. The reduced motivic zeta function of $M$ satisfies the functional equation
\[
	\overline{Z}_M(q^{-1}, T^{-1}) = q^{\rk M - 1}\, \overline{Z}_M(q,T).
\]
\end{theorem}

The proof of \autoref{functionalequationforreducedmotiviczetafunction} relies on the fact that certain polynomials $\overline{P}_M^\cF(q)$, defined in \autoref{poincarepolynomialsection}, are palindromic.  In \autoref{poincarepolynomialsection}, we give a proof of this fact which uses Poincar\'{e} duality for matroid cohomology, in the sense of Adiprasito--Huh--Katz \cite{AdiprasitoHuhKatz}.  The relationship between the duality morphism $\sM_\C \to \sM_\C$ and Poincar\'{e} duality for smooth projective varieties suggests that our functional equation is ``morally'' a consequence of matroid Poincar\'{e} duality.  More precisely, this palindromicity is a consequence of the following relationship between the polynomials $\overline{P}_M^\cF(q)$ and the matroid cohomology algebras defined by Feichtner and Yuzvinsky \cite{FeichtnerYuzvinsky}.  (See \autoref{def:Poincare} and \autoref{def:Coh} for the definitions of $\overline{P}_M (q)$ and the matroid cohomology algebra $D^\bullet(M)$, respectively.)

\begin{theorem}
\label{thm:hilbertseries}
Let $M$ be a loopless matroid. Then
\[
	\overline{P}_M(q^2) = \sum_{i \geq 0} \rk_\Z D^i(M)q^i.
\]
\end{theorem}

We refer the reader to \autoref{thm:hilbertseriesgeneral} for a generalization of \autoref{thm:hilbertseries} for arbitrary building sets.

The functional equation specializes to an identity involving the characteristic polynomials of matroids.  We first set some notation.  If $M$ is a matroid on ground set $E$, then the lattice of flats $\cL$ of $M$ has maximal element $E$ and minimal element $\cl(\emptyset)$.  Let $\cN(M)$ denote the \textbf{order complex} of $\cL_{>\cl(\emptyset)}$.  In other words, $\cN(M)$ is the collection of flags of non-minimal flats of $M$.  We also set
\[
	\cN^*(M) = \{ \cF \in \cN(M) \mid E \in \cF\},
\]
and
\[
	\cN^\circ(M) = \cN(M) \setminus \cN^*(M).
\]

In \cite{DenefLoeser1998, DenefLoeser2002} Denef and Loeser use a specialization of a closely related motivic zeta function to define the motivic nearby fiber and the motivic Milnor fiber (see also \cite[Definition 3.5.3]{DenefLoeser2001} or \cite[Chapter 7 Definition 4.2.3]{CNS}).  Applying this specialization to our functional equation, we obtain the following identity. (See \autoref{defMF} for the definition of $M_{\cF}$.)

\begin{corollary}
\label{alternatingsumofcharactersticpolynomialsgivesbackwardscharacteristicpolynomial}
Let $M$ be a matroid. Then
\[
	\sum_{\cF \in \cN^\circ(M)} (-1)^{\# \cF} \overline{\chi}_{M_{\cF}}(q) = q^{\rk M - 1}\, \overline{\chi}_M(q^{-1}).
\]
\end{corollary}

We refer the reader to \autoref{alternatingsumarbitrarybuildingset} for a generalization of \autoref{alternatingsumofcharactersticpolynomialsgivesbackwardscharacteristicpolynomial} for arbitrary building sets.

\subsection{The topological zeta function}

Our final main results concern the topological zeta function of a matroid, which we define as a certain specialization of our motivic zeta function $Z_M(q,T)$.  The topological zeta function was introduced by Denef and Loeser \cite{DenefLoeser1992} as a ``1-adic'' version of Igusa's local zeta function.
It was only later that Kontsevich introduced motivic integration, and the topological zeta function of a hypersurface was shown to be a specialization of its motivic Igusa zeta function \cite[Section 3.4]{DenefLoeser2001}.  The situation for zeta functions of matroids parallels this story---in \cite{vanderVeer}, van der Veer introduces a topological zeta function for finite, ranked, atomic lattices.  In the case where the lattice is the lattice of flats of a simple matroid, we show that van der Veer's topological zeta function is a specialization of our motivic zeta function.  We refer the reader to \autoref{Sec:Top} for the definition of the specialization $\mu_\topo$.

\begin{definition}
Let $M$ be a matroid. We define the \textbf{topological zeta function} of $M$ to be
\[
	Z_M^\topo(s) = \mu_{\topo}(Z_M(q,T)) \in \Q(s).
\]
\end{definition}

\begin{theorem}
\label{ourtopologicalspecializestovdv}
Let $M$ be a simple matroid with lattice of flats $\cL$, and let $Z_\cL (s)$ be the topological zeta function of $\cL$, in the sense of \cite[Definition 1]{vanderVeer}. Then
\[
	Z_M^{\topo}(s) = Z_\cL(s).
\]
\end{theorem}

Our next result gives the first two coefficients in the Taylor expansion of $Z_M^\topo(s)$ at $s = 0$.  This gives affirmative answers to Questions~1 and~2 in \cite{vanderVeer}.

\begin{theorem}
\label{taylorcoefficientsoftopologicalzetafunction}
Let $M$ be a loopless matroid with ground set $E$. Then
\[
	Z_M^{\topo}(0) = 1,
\]
and
\[
	\left(\frac{d}{ds} Z^{\topo}_M(s)\right) \Big|_{s = 0} = - \# E.
\]
\end{theorem}

The key ingredient in the proof of \autoref{taylorcoefficientsoftopologicalzetafunction} is a recurrence relation for the local motivic zeta function in terms of matroid minors.

\begin{theorem}
\label{recurrencerelationmotiviczetafunctionat0}
Let $M$ be a matroid with ground set $E$ and lattice of flats $\cL$. The local motivic zeta function of $M$ satisfies the recurrence relation
\[
	q^{\rk M}Z_{M}^0(q,T) = \frac{(q-1)q^{-\rk M} T^{\# E}}{1 - q^{-\rk M}T^{\# E}} \left( \barchi_M(q) + \sum_{F \in \widehat\cL} \barchi_{M/F}(q) q^{\rk(M | F)} Z_{M|F}^0 (q, T) \right),
\]
where $\widehat\cL = \cL \setminus \{\cl(\emptyset), E\}$.
\end{theorem}

\begin{acknowledgements}
We thank Christin Bibby and Graham Denham for helpful conversations.  The first author was supported by NSF DMS-1601896.  The third author was supported by NSF DMS-1702428 and an NSF graduate research fellowship.
\end{acknowledgements}

\section{Preliminaries}
\label{preliminaries}

In this section we recall basic notions, set notation, and establish some elementary lemmas. Throughout this section, let $M$ be a matroid with ground set $E$ and lattice of flats $\cL$. Let $\rk_M \colon 2^E \to \Z_{\geq 0}$ denote the rank function of $M$, and let $\mu_M \colon \cL \times \cL \to \Z$ be the M\"obius function of the lattice $\cL$. When the context is clear, we write $\rk$ and $\mu$ instead of $\rk_M$ and $\mu_M$, respectively. We write $\rk M$ for $\rk_M E$.

We shall use standard notation for lattices when working with $\cL$. In particular, if $F$ and $G$ are flats of $M$, then their \textbf{meet} $F \wedge G$ is the flat $F \cap G$ and their \textbf{join} $F \vee G$ is the flat $\cl(F \cup G)$. The maximal element of $\cL$ is $E$ and the minimal element is $\cl(\emptyset)$. We will primarily be concerned with matroids that have no loops, in which case $\cl(\emptyset) = \emptyset$. We let $\widehat{\cL} = \cL \smallsetminus \{E, \cl(\emptyset)\}$ denote the interior of the lattice $\cL$. If $M$ is a loopless matroid, then $\widehat{\cL}$ is the poset consisting of the non-empty proper flats of $M$.

For flats $F_1 \subseteq F_2$, we recall that the interval $[F_1, F_2] = \{F \in \cL \mid  F_1 \subseteq F \subseteq F_2\}$ is isomorphic, via $F \mapsto F\setminus F_1$, to the lattice of flats of the matroid minor $M|F_2/F_1$. By slight abuse of notation, we shall identify flats in the interval $[F_1, F_2]$ with flats in $M| F_2/F_1$. If $\cS \subseteq \cL$ and $F \in \cL$, we will let $\cS_{\leq F}$ denote the set $\{ G \in \cS \mid G \subseteq F \}$. We will use the notation $\cS_{\geq F}$, $\cS_{< F}$, and $\cS_{> F}$ analogously.

\subsection{The characteristic polynomial of a matroid}

Recall that the \textbf{characteristic polynomial} of $M$ is defined as
\[
	\chi_M(q) = \sum_{S \subseteq E} (-1)^{\# S} q^{\rk M - \rk S} \in \Z[q].
\]
If $M$ contains a loop, then $\chi_M(q) = 0$. Otherwise, $\chi_M(q)$ may be written in terms of the M\"{o}bius function $\mu$ as
\[
	\chi_M(q) = \sum_{F \in \cL} \mu(\cl(\emptyset), F) q^{\rk M - \rk F}.
\]
The \textbf{reduced characteristic polynomial} of $M$ is defined as
\[
	\barchi_M(q) = \frac{\chi_M(q)}{q-1} = \sum_{S \subseteq E} (-1)^{\# S} [\rk M - \rk S]_q \in \Z[q],
\]
where for any $n \in \Z_{\geq 0}$, the polynomial $[n]_q \in \Z[q]$ is the \textbf{$q$-analogue} of the integer $n$,
\[
	[n]_q = \frac{q^n-1}{q-1} = 1+q+\cdots+q^{n-1} \in \Z[q].
\]

\begin{remark}
If $\cA$ is a hyperplane arrangement realizing $M$, then $\chi_M(\bL) \in K_0(\Var_\C)$ is equal to the class of the complement of $\cA$, and $\barchi_M(\bL) \in K_0(\Var_\C)$ is equal to the class of the complement of the projectivization of $\cA$. For any $n \in \Z_{\geq 1}$, the class of $\bP^{n-1}$ in $K_0(\Var_\C)$ is equal to $[n]_\bL$.
\end{remark}

The following identity is a consequence of M\"obius inversion.

\begin{proposition} \label{mobius}
For any two flats $F_1 \subseteq F_2$ in $M$,
\[
	[\rk(F_2) - \rk(F_1)]_q = \sum_{F_1 \subseteq F \subsetneq F_2} \barchi_{M|F_2/F}(q).
\]
\end{proposition}

\begin{proof}
For any flats $F_1 \subseteq F_2$ in $M$, the matroid minor $M|F_2/F_1$ is loopless, so
\[
	\chi_{M|F_2/F_1}(q) = \sum_{F_1 \subseteq F \subseteq F_2} \mu(F_1,F) q^{\rk(F_2) - \rk(F)}.
\]
By M\"obius inversion, we have
\[
	q^{\rk(F_2) - \rk(F_1)} = \sum_{F_1 \subseteq F \subseteq F_2} \chi_{M|F_2/F}(q).
\]
To obtain the desired formula, subtract $\chi_{M|F_2/F_2}(q) = 1$ from both sides and divide by $(q-1)$.
\end{proof}

The following special cases of \autoref{mobius} will be of use to us.

\begin{corollary} \label{redcharsub}
Suppose that $M$ is loopless. Then
\[
	\barchi_M(q) = [\rk(M)]_q - \sum_{F \in \widehat{\cL}} \barchi_{M/F}(q).
\]
\end{corollary}

\begin{proof}
This is \autoref{mobius} with $F_1 = \emptyset$ and $F_2 = E$.
\end{proof}

\begin{corollary} \label{atomsum}
Suppose that $M$ is loopless. Then
\[
	\sum_{F \in \widehat{\cL}} (\#F)\, \barchi_{M/F}(q) = (\#E)[\rk(M) - 1]_q.
\]
\end{corollary}

\begin{proof}
Let $e \in E$. Because $M$ is loopless, the flat $\cl(e)$ has rank $1$. Since a flat $F$ contains $e$ if and only if $\cl(e) \subseteq F$, we may apply \autoref{mobius} with $F_1 = \cl(e)$, $F_2 = E$ to see that
\[
	\sum_{e \in F \subsetneq E} \barchi_{M/F}(q) = [\rk(M) - 1]_q.
\]
Therefore,
\begin{align*}
\sum_{F \in \widehat{\cL}} (\# F) \, \barchi_{M/F}(q)
	&= \sum_{F \in \widehat{\cL}} \,\sum_{e \in F} \barchi_{M/F}(q) \\
	&= \sum_{e \in E} \, \sum_{e \in F \subsetneq E} \barchi_{M/F}(q) \\
	&= (\# E) [\rk(M) - 1]_q.
\end{align*}
\end{proof}

\subsection{The Bergman fan of a matroid}
\label{BergmanFan}

\begin{definition}
\label{defWeight}
For $w = (w_e)_{e \in E} \in \R^E$, set
\[
\wt_M(w) = \max_{\text{$B$ basis of $M$}} \sum_{e \in B} w_e.
\]
Let $M_w$ denote the matroid with ground set $E$ whose bases are the bases $B$ of $M$ satisfying $\sum_{e \in B} w_e = \wt_M(w)$. Note that $M_w = M_{w + \lambda \bOne}$ for any $\lambda \in \R$.

For $v \in \R^E/\R\bOne$, set
\[
\overline{\wt}_M(v) = \wt_M(w) - (\rk M) \min(w),
\]
where $w \in \R^E$ is any lift of $v$. Let $M_v$ be the matroid $M_w$ for any lift $w \in \R^E$ of $v$.
\end{definition}

\begin{definition}
\label{defBergmanFan}
The \textbf{Bergman fan} of $M$ is
\[
	\Trop(M) = \{ w \in \R^E \, | \, \text{$M_w$ is loopless}\} \subseteq \R^E.
\]
\end{definition}

We now describe two unimodular fans determined by a matroid $M$ on $E$. For any $S \subseteq E$, let $v_S \in \R^E$ denote the indicator vector of $S$.  That is, the $e$th coordinate of $v_S$ is 1 when $e \in S$ and 0 when $e \in E \setminus S$.  For any $\cS \subseteq 2^E$, let $\sigma_{\cS} \subseteq \R^E$ denote the cone
\[
	\sigma_{\cS} = \sum_{S \in \cS} \R_{\geq 0} v_S.
\]
Let $\Sigma(M)$ be the fan
\[
	\Sigma(M) = \{\sigma_\cF \, | \, \cF \in \cN(M)\},
\]
and let $\Sigma^\circ(M)$ be
\[
	\Sigma^\circ(M) = \{\sigma_\cF \, | \, \cF \in \cN^\circ(M)\}.
\]

It is straightforward to check that $\Sigma(M)$ and $\Sigma^\circ(M)$ are unimodular fans. The supports of $\Sigma(M)$ and $\Sigma^\circ(M)$ were given by Ardila and Klivans.

\begin{theorem}\cite[Theorem~1]{ArdilaKlivans}
Suppose that $M$ is loopless.  Then the support of $\Sigma(M)$ is
\[
	\bigcup_{\cF \in \cN(M)} \sigma_\cF = \Trop(M) \cap \R_{\geq 0}^E,
\]
and the support of $\Sigma^\circ(M)$ is
\[
	\bigcup_{\cF \in \cN^\circ(M)} \sigma_\cF = \Trop(M) \cap \partial\R_{\geq 0}^E,
\]
where $\partial \R_{\geq 0}^E = \R_{\geq 0}^E \setminus \R_{>0}^E$ is the boundary of the positive orthant of $\R^E$.
\end{theorem}

For each flag $\cF \in \cN(M)$, the function $w \mapsto M_w$ is constant on $\relint(\sigma_\cF)$ \cite[Proposition 1]{ArdilaKlivans}.  We therefore may make the following definition.

\begin{definition}
\label{defMF}
For $\cF \in \cN(M)$, we let $M_\cF$ denote the matroid $M_w$ for any $w \in \relint(\sigma_\cF)$.
\end{definition}

We will frequently use the fact that, for any $\cF \in \cN^\circ(M)$, we have an equality of matroids $M_\cF = M_{\cF \cup \{E\}}$.  Throughout, we let $z_\cF(F) = \max (\cF_{< F})$.  We state a useful result about matroids of the form $M_\cF$.

\begin{proposition}\cite[Propositions~1 and~2]{ArdilaKlivans}
\label{initialmatroiddecompositionmatroidminorsmaximalbuildingset}
Suppose that $M$ is loopless, and let $\cF \in \cN^*(M)$.  Then
\[
	M_\cF = \bigoplus_{F \in \cF} M|F/z_\cF(F).
\]
\end{proposition}

In particular, \autoref{initialmatroiddecompositionmatroidminorsmaximalbuildingset} implies that for each $\cF \in \cN(M)$, the characteristic polynomial $\chi_{M_\cF}(q)$ is divisible by $(q-1)^{\# \cF}$.  Similarly, for each $\cF \in \cN^\circ(M)$ we have $M_{\cF} = M_{\cF \cup \{ E \}}$, and so the reduced characteristic polynomial $\barchi_{M_\cF}(q)$ is divisible by $(q-1)^{\# \cF}$.

\begin{remark}
In the statement of \autoref{initialmatroiddecompositionmatroidminorsmaximalbuildingset}, the ground set of the matroid minor $M|F/z_\cF(F)$ is identified in the standard way with $F \setminus  z_\cF(F) \subseteq E$.  This identification will be implicit in several of the statements below.
\end{remark}

We will also need the following refinement of \autoref{initialmatroiddecompositionmatroidminorsmaximalbuildingset}.

\begin{proposition}
\label{decomposeinitialdegenerationofrefinementmaximalbuildingset}
Suppose that $M$ is loopless, and let $\cF, \cG \in \cN^*(M)$ such that $\cF \subseteq \cG$. Then
\[
	M_\cG = \bigoplus_{F \in \cF} (M | F / z_{\cF}(F))_{\cG_F},
\]
where $\cG_F \in \cN^*(M | F /z_\cF(F))$ is the flag
\[
	\{ G \in \cG \mid z_\cF(F) \subsetneq G \subseteq F \}.
\]
\end{proposition}

\begin{proof}
By \autoref{initialmatroiddecompositionmatroidminorsmaximalbuildingset},
\begin{align*}
	M_\cG = \bigoplus_{G \in \cG} M | G/ z_{\cG}(G)
	&= \bigoplus_{F \in \cF} \bigoplus_{G \in \cG_F} (M | F / z_\cF(F)) | G / z_{\cG_F}(G)\\
	&= \bigoplus_{F \in \cF} \ (M | F / z_{\cF}(F))_{\cG_F}.
\end{align*}
\end{proof}

\subsection{Wonderful models and the Bergman fan}

In the realizable case, the Bergman fan and $\Sigma(M)$ are closely related to the wonderful models of De Concini and Procesi introduced in \cite{DeConciniProcesi}.  Because we do not assume that our matroids are realizable, the ideas in this subsection will not be used directly in any of our proofs.  Nevertheless, \autoref{Prop:Poincare} suggests the definitions of Poincar\'{e} and Euler-Poincar\'{e} polynomials given in \autoref{def:Poincare} and \autoref{def:EP} below.

Suppose that $M$ is loopless.  The fan $\Sigma(M)$ defines a toric variety with dense torus $\bG_m^E$, and the image $\overline{\Sigma}(M)$ of $\Sigma(M)$ in $\R^E/\R v_E$ defines a toric variety with dense torus $\bG_m^E/\bG_m$, where the quotient is taken with respect to $\bG_m$ acting on $\bG_m^E$ by the diagonal action.

Now let $\cA$ be a hyperplane arrangement realizing $M$, and let $Y$ (resp. $\overline{Y}$) be the wonderful model of the complement of $\cA$ (resp. the complement of the projectivization of $\cA$), with respect to the maximal building set. Embedding by any choice of linear forms defining the hyperplanes in $\cA$, we obtain a linear subspace $V \subset \bA^E$ such that the hyperplanes in $\cA$ are the intersections of $V$ with the coordinate hyperplanes of $\bA^E$. The complement of $\cA$ is equal to $V \cap \bG_m^E$, where $\bG_m^E$ is naturally identified with the complement of the coordinate hyperplanes of $\bA^E$, and the complement of the projectivization of $\cA$ is equal to $\bP(V) \cap (\bG_m^E / \bG_m)$, where $(\bG_m^E / \bG_m)$ is naturally identified with the complement of the coordinate hyperplanes in $\bP(\bA^E)$. Then $Y$ is isomorphic to the closure of $V \cap \bG_m^E$ in the toric variety defined by $\Sigma(M)$, and $\overline{Y}$ is isomorphic to the closure of $\bP(V) \cap (\bG_m^E / \bG_m)$ in the toric variety defined by $\overline{\Sigma}(M)$. Under these identifications, $\overline{Y}$ has the structure of a tropical compactification, in the sense of Tevelev \cite{Tevelev}. Similarly, if we set
\[
	\widetilde{\Sigma}(M) = \Sigma(M) \cup \{ \sigma_\cF + \R_{\geq 0} (-v_E) \mid \cF \in \cN^\circ(M)\},
\]
then $Y$ is an open subset of a tropical compactification in the toric variety defined by $\widetilde{\Sigma}(M)$. In fact, $Y$ is the intersection of this tropical compactification with the torus-invariant open subvariety defined by $\Sigma(M)$.

Under these identifications, if $\cF \in \cN(M)$ (resp. $\cF \in \cN^\circ(M)$) and $Y_\cF$ (resp. $\overline{Y}_\cF$) is the corresponding locally closed stratum in $Y$ (resp. $\overline{Y}$), then $Y_\cF$ (resp. $\overline{Y}_\cF$) is equal to the intersection of $Y$ (resp. $\overline{Y}$) with the torus orbit corresponding to $\sigma_\cF \in \Sigma(M)$ (resp. the image of $\sigma_\cF$ in $\R^E/\R v_E)$. Thus by \cite[Lemma 3.6]{HelmKatz}, we have
\[
	Y_\cF \times \bG_m^{\# \cF} \cong \init_w(V \cap \bG_m^E), \qquad \overline{Y}_\cF \times \bG_m^{\# \cF} \cong \init_v(\bP(V) \cap (\bG_m^E / \bG_m)),
\]
where $\init_w$ (resp. $\init_v$) denotes taking initial degeneration with respect to any $w \in \relint (\sigma_\cF)$ (resp. any $v$ in the relative interior of the image of $\sigma_\cF$ in $\R^E/\R v_E$). Therefore because $[Y_\cF] \in K_0(\Var_\C)$ (resp. $[\overline{Y}_\cF] \in K_0(\Var_\C)$) is a polynomial in $\bL$, we obtain the following proposition computing the classes in $K_0(\Var_\C)$ of the locally closed strata of $Y$ and $\overline{Y}$.

\begin{proposition}
\label{Prop:Poincare}
If $\cF \in \cN(M)$ (resp. $\cF \in \cN^\circ(M))$, then
\[
	[Y_\cF] = \left(\frac{\chi_{M_\cF}(q)}{(q-1)^{\# \cF}} \right) \Big|_{q = \bL} \qquad \left( \text{resp. } [\overline{Y}_\cF] = \left(\frac{\overline{\chi}_{M_\cF}(q)}{(q-1)^{\# \cF}} \right )  \Big|_{q = \bL} \right).
\]
\end{proposition}

\section{Rationality of the motivic zeta functions}
\label{rationalitysection}

Let $\Z[q^{\pm 1}] \llbracket T \rrbracket_\rat$ be the $\Z[q^{\pm 1}]$-subalgebra of $\Z[q^{\pm 1}]\llbracket T \rrbracket$ generated by elements of the form
\[ (q-1) \frac{q^b T^a}{1-q^bT^a} = (q-1) \sum_{k = 1}^{\infty} q^{kb}T^{ka}, \]
for $a \in \Z_{>0}$ and $b \in \Z$.  In this section, we will prove the following thoerem, which gives expressions for our motivic zeta functions as elements of $\Z[q^{\pm 1}] \llbracket T \rrbracket_\rat$.

\begin{theorem}
\label{rationalformulaformotiviczetafunctionmaximalbuildingset}
Let $M$ be a matroid. Then
\[
	Z_M(q,T) = q^{-\rk M} \sum_{\cF \in \cN(M)} \frac{\chi_{M_{\cF}}(q)}{(q-1)^{\#\cF}} \prod_{F \in \cF } (q-1) \frac{q^{-\rk F}T^{\#F}}{1 - q^{-\rk F}T^{\#F}},
\]
\[
	Z_M^0(q,T) = q^{-\rk M} \sum_{\cF \in \cN^*(M)} \frac{\chi_{M_{\cF}}(q)}{(q-1)^{\#\cF}} \prod_{F \in \cF } (q-1) \frac{q^{-\rk F}T^{\#F}}{1 - q^{-\rk F}T^{\#F}},
\]
and
\[
	\overline{Z}_M(q,T) = q^{-(\rk M-1)} \sum_{\cF \in \cN^\circ(M)} \frac{\overline{\chi}_{M_{\cF}}(q)}{(q-1)^{\#\cF}} \prod_{F \in \cF} (q-1) \frac{q^{-\rk F}T^{\#F}}{1 - q^{-\rk F}T^{\#F}}.
\]
\end{theorem}

\begin{remark}
\label{rmk:realized}
Let $\cA$ be a hyperplane arrangement realizing $M$. Let $Y$ be the wonderful model of $\cA$ with respect to the maximal building set, and let $\overline{Y}$ be the wonderful model of the projectivization of $\cA$ with respect to the maximal building set. The maps $Y \to \bA^{\rk M}$ and $\overline{Y} \to \bP(\bA^{\rk M})$ are log resolutions of $\cA$ and the projectivization of $\cA$, respectively, and applying \cite[Corollary 3.3.2]{DenefLoeser2001} to these resolutions gives \autoref{rationalformulaformotiviczetafunctionmaximalbuildingset} in the case where $M$ is realizable.
\end{remark}

Before we prove \autoref{rationalformulaformotiviczetafunctionmaximalbuildingset}, we give formulas relating $Z_M(q,T)$, $Z_M^0(q,T)$, and $\overline{Z}_M(q,T)$.

\begin{proposition}
\label{relationshipbetweenvariousmotiviczetafunctions}
Let $M$ be a matroid with ground set $E$. Then
\[
	Z_M(q,T) = q^{-1}(q-1) \left( \frac{1}{1 - q^{-\rk M}T^{\# E}} \right) \overline{Z}_M(q,T),
\]
and
\[
	Z_M^0(q,T) = q^{-1} (q-1) \left( \frac{q^{-\rk M} T^{\# E}}{1-q^{-\rk M}T^{\# E}} \right) \overline{Z}_M(q,T).
\]
\end{proposition}

\begin{proof}
This follows immediately from \autoref{rationalformulaformotiviczetafunctionmaximalbuildingset}.  Alternatively, this proposition can be shown directly using \autoref{boundaryoforthant} below.
\end{proof}

\subsection{Proof of \autoref{rationalformulaformotiviczetafunctionmaximalbuildingset}}

Let $M$ be a matroid with ground set $E$. If $M$ has a loop, then both sides of the equations in \autoref{rationalformulaformotiviczetafunctionmaximalbuildingset} are equal to zero. Therefore, for the remainder of \autoref{rationalitysection}, we will assume that $M$ is loopless.

We begin with some notation that will be required for the proof of rationality.

\begin{definition}
\label{defVert}
For $w = (w_e)_{e \in E} \in \R^E$, set
\[
|w| = \sum_{e \in E} w_e.
\]
For $v \in \R^E/\R\bOne$ with lift $w \in \R^E$, set
\[
\overline{|v|} = |w| - (\# E) \min(w).
\]
\end{definition}

\begin{remark} \label{boundaryoforthant}
The projection map $\pi \colon \R^E \to \R^E/\R v_E$ induces a bijection from $\Z^E \cap \partial \R_{\geq 0}^E$ to $\Z^E/\Z v_E$.  For each $w \in \Z^E \cap \partial \R_{\geq 0}^E$ we have
\[
|w| = \overline{|\pi(w)|} \in \Z_{\geq 0},
\]
and
\[
\wt_M(w) = \overline{\wt}_M(\pi(w)) \in \Z_{\geq 0}.
\]
\end{remark}

\begin{lemma}
\label{weightislinearoncone}
Let $\sigma$ be a cone in $\R^E$ such that $w \mapsto M_w$ is constant on $\relint(\sigma)$. Then there exists a linear function $\R^E \to \R$ that coincides with $\wt_M$ on $\sigma$.
\end{lemma}

\begin{proof}
Let $M_\sigma$ be the matroid such that $M_\sigma = M_w$ for all $w \in \relint(\sigma)$.  Let $B \subseteq E$ be a basis of $M_\sigma$, and consider the linear function
\[
	\wt_B: \R^E \to \R: (w_e)_{e \in E} \mapsto \sum_{e \in B} w_e.
\]
Then $\wt_M(w) = \wt_B(w)$ for all $w \in \relint(\sigma)$. The lemma then follows from the fact that $\wt_M$ and $\wt_B$ are continuous.
\end{proof}

We now prove \autoref{rationalformulaformotiviczetafunctionmaximalbuildingset}.

\begin{proof}[Proof of \autoref{rationalformulaformotiviczetafunctionmaximalbuildingset}]
Because $\Sigma(M)$ is a unimodular fan supported on $\Trop(M)\cap\R^E_{\geq 0}$, we have
\begin{align*}
	Z_M(q, T) &= q^{-\rk M} \sum_{w \in \Trop(M) \cap \Z_{\geq 0}^E} \chi_{M_w}(q) q^{-\wt_M(w)}T^{|w|}\\
	&= q^{-\rk M} \sum_{\cF \in \cN(M)} \sum_{w \in \relint(\sigma_\cF) \cap \Z^E} \chi_{M_w}(q) q^{-\wt_M(w)}T^{|w|}\\
	&= q^{-\rk M} \sum_{\cF \in \cN(M)} \chi_{M_\cF}(q) \sum_{(k_F) \in \Z_{>0}^\cF} q^{-\wt_M( \sum_{F} k_F v_F)} T^{|\sum_{F} k_F v_F|}.
\end{align*}
Because
\[
	\Trop(M) \cap \R_{>0}^E = \bigcup_{\cF \in \cN^*(M)} \relint(\sigma_\cF),
\]
we also get that
\[
	Z_{M}^0(q,T) = q^{-\rk M} \sum_{\cF \in \cN^*(M)}  \chi_{M_\cF}(q) \sum_{(k_F) \in \Z_{>0}^\cF} q^{-\wt_M( \sum_{F} k_F v_F)} T^{|\sum_{F } k_F v_F|}.
\]
Similarly, by \autoref{boundaryoforthant} and the fact that $\Sigma^\circ(M)$ is a unimodular fan supported on $\Trop(M) \cap \partial \R_{\geq 0}^E$, we have
\begin{align*}
	\overline{Z}_M(q,T) &= q^{-(\rk M - 1)} \sum_{w \in \Trop(M) \cap \Z^E \cap \partial \R_{\geq 0}^E} \barchi_{M_w}(q) q^{-\wt_M(w)} T^{|w|}\\
	&= q^{-(\rk M-1)} \sum_{\cF \in \cN^\circ(M)} \barchi_{M_\cF}(q) \sum_{(k_F)_{F} \in \Z_{>0}^\cF} q^{-\wt_M( \sum_{F} k_F v_F)} T^{|\sum_{F} k_F v_F|}.
\end{align*}
Note that the rightmost sum in all three expressions is the same.  It therefore suffices to compute this sum.  By \autoref{weightislinearoncone}, for each $\cF \in \cN(M)$, we have
\begin{align*}
	\sum_{(k_F) \in \Z_{>0}^\cF} q^{-\wt_M( \sum_{F} k_F v_F)} T^{|\sum_{F} k_F v_F|} &= \prod_{F \in \cF} \sum_{k = 1}^{\infty} \left( q^{-\wt_M(v_F)} T^{|v_F|} \right)^k\\
	&= \prod_{F \in \cF} \frac{q^{-\wt_M(v_F)} T^{|v_F|}}{1-q^{-\wt_M(v_F)} T^{|v_F|}}.
\end{align*}
\autoref{rationalformulaformotiviczetafunctionmaximalbuildingset} then follows from the fact that for each flat $F$ of $M$,
\[
	\wt_M(v_F) = \max_{\text{$B$ basis of $M$}} \#(B \cap F) = \rk F,
\]
and
\[
	|v_F| = \# F.
\]
\end{proof}

\section{The Poincar\'{e} polynomial of a matroid}
\label{poincarepolynomialsection}

In this section, we define certain polynomials that can be associated to a matroid, and prove that these polynomials are palindromic.  This will be an essential ingredient in the proof of \autoref{functionalequationforreducedmotiviczetafunction}.  In the course of the proof, we establish \autoref{thm:hilbertseries}, which gives a formula for the Hilbert series of the cohomology ring of a matroid, in the sense of Feichtner and Yuzvinsky \cite{FeichtnerYuzvinsky}.

We begin by defining Poincar\'{e} polynomials of matroids.

\begin{definition}
\label{def:Poincare}
Let $M$ be a matroid. For each $\cF \in \cN^\circ(M)$, we define the \textbf{Poincar\'{e} polynomial of $M$ at $\cF$} to be
\[
	\overline{P}_M^\cF(q) = \sum_{\substack{\cG \in \cN^\circ(M) \\ \cG \supseteq \cF}} \frac{\overline{\chi}_{M_\cG}(q)}{(q-1)^{\# \cG}} \in \Z[q].
\]
We define the \textbf{Poincar\'{e} polynomial} of $M$ to be
\[
	\overline{P}_M(q) = \sum_{\cF \in \cN^\circ(M)} \frac{\overline{\chi}_{M_\cF}(q)}{(q-1)^{\# \cF}} \in \Z[q].
\]
In other words, $\overline{P}_M(q)$ is the Poincar\'{e} polynomial of $M$ at the empty flag.
\end{definition}

Note that $\overline{P}_M^{\cF}(q)$ is a polynomial of degree $\rk M - 1 - \# \cF$.

\begin{remark}
\label{realizablegeometrypoincarepolynomial}
Let $\cA$ be a hyperplane arrangement realizing $M$. Let $\overline{Y}$ be the wonderful model, with respect to the maximal building set, of the projectivization of $\cA$.  Let $\overline{D}_{\cF}$ be the closure of the stratum in $\overline{Y}$ corresponding to $\cF$.  By \autoref{Prop:Poincare}, $\overline{P}_M^\cF(\bL) = [\overline{D}_\cF] \in K_0(\Var_\C)$.  Because $\overline{D}_\cF$ is smooth and projective, $\overline{P}_M^\cF(q^2)$ is the Poincar\'{e} polynomial of $\overline{D}_\cF$. In particular, $\overline{P}_M(q^2)$ is the Poincar\'{e} polynomial of $\overline{Y}$.
\end{remark}

We now recall the definition of the cohomology ring of a matroid (with respect to the maximal building set).

\begin{definition}\cite{FeichtnerYuzvinsky}
\label{def:Coh}
Let $M$ be a matroid with lattice of flats $\cL$. The \textbf{cohomology ring} of $M$ is the graded algebra
\[
	D^\bullet(M) = \bigoplus_i D^i(M) =  \Z[ \{x_F\}_{F \in \cL_{>\cl(\emptyset)}}] / \mathcal{I},
\]
where each $x_F$ has degree 2, and $\mathcal{I}$ is the ideal generated by
\[
	\prod_{i = 1}^k x_{F_i}, \quad \text{for $\{F_1, \dots, F_k\} \notin \cN(M)$},
\]
and
\[
	\sum_{F \supseteq A}x_F, \quad \text{for $A$ an atom of $\cL$}.
\]
\end{definition}

\begin{remark}
\label{realizablegeometryfeichtneryuzvinskycohomologyring}
Let $\cA$ be a hyperplane arrangement realizing $M$. Let $\overline{Y}$ be the wonderful model, with respect to the maximal building set, of the projectivization of $\cA$. Then $D^\bullet(M)$ is isomorphic to the cohomology ring of $\overline{Y}$ \cite[Section 4 Corollary 2]{FeichtnerYuzvinsky}.
\end{remark}

\begin{remark}
In the case where $M$ is realizable, \autoref{thm:hilbertseries} follows from \autoref{realizablegeometrypoincarepolynomial} and \autoref{realizablegeometryfeichtneryuzvinskycohomologyring}.
\end{remark}

Before proving \autoref{thm:hilbertseries}, we prove that it implies palindromicity of the Poincar\'{e} polynomials.

\begin{theorem}
\label{numericalPoincaredualitymaximalbuildingset}
Let $M$ be a loopless matroid. For each $\cF \in \cN^\circ(M)$,
\[
	\overline{P}_M^\cF(q^{-1}) = q^{-(\rk M - 1 - \# \cF)} \overline{P}_M^\cF(q)  \in \Z[q^{\pm 1}].
\]
\end{theorem}

\begin{proof}
When $\cF = \emptyset$, this follows from \autoref{thm:hilbertseries} and \cite[Theorem 6.19]{AdiprasitoHuhKatz}, which states that $D^\bullet(M)$ satisfies a version of Poincar\'{e} duality.

Now let $\cF \in \cN^\circ(M)$, and let $E$ be the ground set of $M$. For each $F \in \cF \cup \{E\}$, let $M_F$ denote the matroid minor $M| F /z_{\cF \cup \{E\}}(F)$. For each $F \in \cF \cup \{E\}$ and $\cG \in \cN^\circ(M)$ such that $\cG \supseteq \cF$, let $\cG_F \in \cN^\circ(M_F)$ denote the flag
\[
	\{G \in \cG \mid z_{\cF \cup \{E\}}(F) \subsetneq G \subsetneq F\}.
\]
Then by \autoref{decomposeinitialdegenerationofrefinementmaximalbuildingset},
\begin{align*}
	\frac{\barchi_{M_\cG}(q)}{(q-1)^{\# \cG}} &= \frac{\chi_{M_{\cG \cup \{E\}}}(q)}{(q-1)^{\# \cG \cup \{E\}}} = \prod_{F \in \cF \cup \{E\}} \frac{\chi_{(M_F)_{\cG_F \cup \{F\}}}(q)}{(q-1)^{\#(\cG_F \cup \{F\})}}\\
	&= \prod_{F \in \cF \cup \{E\}} \frac{\barchi_{(M_F)_{\cG_F}}(q)}{(q-1)^{\# \cG_F}}.
\end{align*}
Therefore,
\begin{align}
	\overline{P}_M^\cF(q) &= \sum_{\substack{\cG \in \cN^\circ(M) \\ \cG \supseteq \cF}} \frac{\overline{\chi}_{M_\cG}(q)}{(q-1)^{\# \cG}} = \prod_{F \in \cF \cup \{E\}} \, \sum_{\cG \in \cN^\circ(M_F)} \frac{\barchi_{(M_F)_\cG}(q)}{(q-1)^{\# \cG}}\nonumber\\
	&= \prod_{F \in \cF \cup \{E\}} \overline{P}_{M_F}(q). \label{poincareatflagproductofpoincare}
\end{align}
Thus by the case where the flag is empty,
\begin{align*}
	\overline{P}_M^\cF(q^{-1}) &= \prod_{F \in \cF \cup \{E\}} \overline{P}_{M_F}(q^{-1}) = \prod_{F \in \cF \cup \{E\}} q^{- (\rk M_F - 1)} \overline{P}_{M_F}(q) \\
	&= q^{-(\rk M - 1 - \# \cF)} \overline{P}_M^\cF(q) ,
\end{align*}
where we note that \autoref{initialmatroiddecompositionmatroidminorsmaximalbuildingset} implies that
\[
	\rk M = \rk M_{\cF \cup \{E\}} = \sum_{F \in \cF \cup \{E\}} \rk M_F.
\]
\end{proof}

\begin{remark}
When $M$ is realizable, \autoref{numericalPoincaredualitymaximalbuildingset} is a consequence of \autoref{realizablegeometrypoincarepolynomial}. In general, \autoref{numericalPoincaredualitymaximalbuildingset} states that $\overline{P}_M^\cF(q)$ behaves under $q \mapsto q^{-1}$ like the class of a smooth projective variety of dimension $\rk M - 1 - \# \cF$.
\end{remark}

\subsection{Proof of \autoref{thm:hilbertseries}} \label{proofofpoincarecohomologysection}

For a matroid $M$, let $H_M(q)$ be the polynomial
\[
	H_M(q) = \sum_{\cF \in \cN(M)} \prod_{F \in \cF} \left( [\rk F - \rk z_\cF(F)]_q - 1 \right) \in \Z[q].
\]
Feichtner and Yuzvinsky show \cite[Section 3 Corollary 1]{FeichtnerYuzvinsky} that for any matroid $M$, a $\Z$-basis of $D^\bullet(M)$ is given by the set of monomials of the form
\[
	\prod_{F \in \cF} x_F^{m_F},
\]
where $\cF \in \cN(M)$ and each $m_F \in \{1, \dots, \rk F - \rk z_\cF(F) - 1\}$. Using this monomial basis, it is straightforward to see that
\[
	\sum_{i \geq 0} \rk_\Z D^i(M)q^i = H_M(q^2).
\]
To prove \autoref{thm:hilbertseries}, it suffices to show that $\overline{P}_M(q) = H_M(q)$ when $M$ is loopless. To do this, we will show in the next two lemmas that both of these polynomials satisfy the same recursive formula.

\begin{lemma}
\label{recursionforpoincarepolynomials}
Let $M$ be a matroid with lattice of flats $\cL$. Then
\[
	\overline{P}_M(q) = \barchi_M(q) + \sum_{F \in \widehat\cL} \barchi_{M/F}(q)\, \overline{P}_{M|F}(q).
\]
\end{lemma}

\begin{proof}
If $M$ has a loop, then both sides are equal to 0, so we may assume that $M$ is loopless. Let $E$ be the ground set of $M$. If $\cF \in \cN^\circ(M)$, then the maximal element of $\cF$ is a proper flat. Thus, \autoref{initialmatroiddecompositionmatroidminorsmaximalbuildingset} implies
\begin{align*}
	M_\cF &= \bigoplus_{F \in \cF \cup \{E\}} M | F / z_{\cF \cup \{E\}}(F) \\
	&= (M / \max \cF) \oplus \bigoplus_{F \in \cF} M | F / z_{\cF}(F) \\
	&= (M / \max \cF) \oplus (M | \max \cF)_{\cF \setminus \{\max \cF\}},
\end{align*}
where by slight abuse of notation, we let $\cF \setminus \{\max \cF\}$ denote the corresponding element of $\cN^\circ(M | \max\cF)$. Thus for any $\cF \in \cN^\circ(M)$,
\[
	\barchi_{M_\cF}(q) = \left( \barchi_{M/\max\cF}(q)\right) \left(\barchi_{(M | \max \cF)_{\cF \setminus \{\max \cF\}}}(q) \right) (q-1).
\]
Now, in the sum defining $\overline{P}_M(q)$, we sort flags by their maximal element:
\begin{align*}
	\overline{P}_M(q) &= \sum_{\cF \in \cN^\circ(M)} \frac{\overline{\chi}_{M_\cF}(q)}{(q-1)^{\# \cF}} = \barchi_M(q) + \sum_{F \in \widehat\cL} \, \sum_{\substack{\cF \in \cN^\circ(M)\\ \max \cF = F}} \frac{\overline{\chi}_{M_\cF}(q)}{(q-1)^{\# \cF}}\\
	&= \barchi_M(q) + \sum_{F \in \widehat\cL} \barchi_{M/F}(q) \sum_{\cF \in \cN^\circ(M | F)}  \frac{\overline{\chi}_{(M|F)_\cF}(q)}{(q-1)^{\# \cF}}\\
	&= \barchi_M(q) + \sum_{F \in \widehat\cL} \barchi_{M/F}(q)\, \overline{P}_{M|F}(q).
\end{align*}
\end{proof}

\begin{lemma}
\label{recursionforhilbertseries}
Let $M$ be a loopless matroid with lattice of flats $\cL$. Then
\[
	H_M(q) = \barchi_M(q) + \sum_{F \in \widehat\cL} \barchi_{M/F}(q)\, H_{M|F}(q).
\]
\end{lemma}

\begin{proof}
Let $E$ be the ground set of $M$. By \autoref{redcharsub},
\begin{align*}
	&\barchi_M(q) + \sum_{F \in \widehat\cL} \barchi_{M/F}(q)\, H_{M|F}(q) = [\rk M]_q + \sum_{F \in \widehat\cL} \barchi_{M/F}(q)\left( H_{M|F}(q) - 1 \right)\\
	&= [\rk M]_q + \sum_{F \in \widehat\cL} \barchi_{M/F}(q) \sum_{\cF \in \cN(M | F) \setminus \{\emptyset\}} \,\prod_{G \in \cF} \left( [ \rk G - \rk z_\cF(G)]_q - 1\right) \\
	&= [\rk M]_q + \sum_{\cF \in \cN^\circ(M) \setminus \{\emptyset\}} \sum_{\substack{F \in \cL \\ \max\cF \subseteq F \subsetneq E}} \barchi_{M/F}(q) \prod_{G \in \cF} \left( [ \rk G - \rk z_\cF(G) ]_q - 1\right).
\end{align*}
By \autoref{mobius}, the above string of equalities continues
\begin{align*}
	&= [\rk M]_q + \sum_{\cF \in \cN^\circ(M) \setminus \{\emptyset\}} [\rk M - \rk(\max \cF)]_q \prod_{G \in \cF} \left( [ \rk G - \rk z_\cF(G) ]_q - 1\right)\\
	&= H_M(q),
\end{align*}
where the last equality can be seen by rewriting
\[
	[\rk M]_q = 1+([\rk E - \rk z_{\{E\}}(E)]_q - 1),
\]
and
\[
	[\rk E - \rk(\max \cF)]_q = 1 + ([\rk E - \rk z_{\cF \cup \{E\}} (E)]_q - 1)
\]
for each $\cF \in \cN^\circ(M) \setminus \{\emptyset\}$.
\end{proof}

The following corollary completes our proof of \autoref{thm:hilbertseries}.

\begin{corollary}
\label{poincarepolynomialrelatedtomonomialbasis}
Let $M$ be a loopless matroid. Then
\[
	\overline{P}_M(q) = H_M(q) = \sum_{\cF \in \cN(M)} \prod_{F \in \cF} \left( [\rk F - \rk z_\cF(F)]_q - 1 \right).
\]
\end{corollary}

\begin{proof}
This follows from \autoref{recursionforpoincarepolynomials}, \autoref{recursionforhilbertseries}, and induction on the rank of $M$.
\end{proof}

\section{A functional equation for the reduced motivic zeta function}
\label{proofoffunctionalequationsection}

In this section, we prove \autoref{functionalequationforreducedmotiviczetafunction}.  As a consequence, we obtain \autoref{alternatingsumofcharactersticpolynomialsgivesbackwardscharacteristicpolynomial} as well.  If the matroid $M$ has a loop, then both sides of the equations in \autoref{functionalequationforreducedmotiviczetafunction} and \autoref{alternatingsumofcharactersticpolynomialsgivesbackwardscharacteristicpolynomial} are equal to 0.  We therefore assume in this section that $M$ is a loopless matroid.

\autoref{functionalequationforreducedmotiviczetafunction} is a matroid analogue of the functional equation \cite[Chapter 7 Proposition 3.3.10]{CNS}, which holds for the motivic Igusa zeta function of a subscheme of a smooth projective variety. Our proof is based on the proof in \cite{CNS}.  For readers familiar with this argument, we note the following adaptations, which are necessary in the setting of non-realizable matroids.

\begin{itemize}

\item The formula for $\overline{Z}_M(q,T)$ given in \autoref{rationalformulaformotiviczetafunctionmaximalbuildingset} takes the place of the formula for a motivic Igusa zeta function in terms of a log resolution.

\item Each Poincar\'{e} polynomial $\overline{P}_M^\cF(q)$ plays the role of the closure of a stratum in a log resolution.

\item The class of a locally closed stratum in a log resolution is replaced by the polynomial $\barchi_{M_\cF}(q)/(q-1)^{\# \cF}$. In the geometric setting, the class of a locally closed stratum of a log resolution can be written as an alternating sum of closures of strata. A combinatorial analogue of this statement is given by applying M\"{o}bius inversion, on the poset of flags of flats, to the definition of the Poincar\'{e} polynomials. Explicitly, if $\cF \in \cN^\circ(M)$, then
\begin{align}
\label{eq:MobiusPoincare}
	\frac{\barchi_{M_\cF}(q)}{(q-1)^{\# \cF}} =  \sum_{\substack{\cG \in \cN^\circ(M) \\ \cG \supseteq \cF}}(-1)^{\#\cG - \#\cF}\, \overline{P}_M^\cG(q).
\end{align}

\end{itemize}

Note that $\Z[q^{\pm 1}]\llbracket T \rrbracket_\rat$ is contained in the field $\Q(q)(T)$, where the latter is considered as a subfield of $\Q(q)\llbracket T \rrbracket [T^{-1}]$. Thus for any $f(q,T) \in \Z[q^{\pm 1}]\llbracket T \rrbracket_\rat$, there is a well defined $f(q,T^{-1}) \in \Q(q)(T)$.  It suffices to prove that \autoref{functionalequationforreducedmotiviczetafunction} holds in the field $\Q(q)(T)$, so throughout our proof, we will freely manipulate our expressions as rational functions. We begin with the following lemma.

\begin{lemma}
\label{rewritingreducedmotiviczetafunctionintermsofchowpolynomials}
Let $M$ be a matroid. We have
\[
	\overline{Z}_M(q,T) = q^{-(\rk M-1)} \sum_{\cF \in \cN^\circ(M)} (-1)^{\#\cF}\, \overline{P}_M^\cF(q) \prod_{F \in \cF} \frac{1- q^{-(\rk F - 1)} T^{\# F}}{1-q^{-\rk F} T^{\# F}}.
\]
\end{lemma}

\begin{proof}
For each $\cG \in \cN^\circ(M)$,
\[
	\sum_{\substack{\cF \in \cN^\circ(M)\\ \cF \subseteq \cG}}(-1)^{\# \cF } \prod_{F \in \cF} \frac{(q-1)q^{-\rk F}T^{\#F}}{1 - q^{-\rk F}T^{\#F}} = \prod_{F \in \cG} \left( 1 - \frac{(q-1)q^{-\rk F}T^{\#F}}{1 - q^{-\rk F}T^{\#F}}\right).
\]
By \autoref{rationalformulaformotiviczetafunctionmaximalbuildingset},
\begin{align*}
	\overline{Z}_M(q,T) &= q^{-(\rk M-1)} \sum_{\cF \in \cN^\circ(M)} \frac{\overline{\chi}_{M_{\cF}}(q)}{(q-1)^{\#\cF}} \prod_{F \in \cF} (q-1) \frac{q^{-\rk F}T^{\#F}}{1 - q^{-\rk F}T^{\#F}} .
\end{align*}
By \eqref{eq:MobiusPoincare}, this is equal to
\begin{align*}
	& q^{-(\rk M-1)} \sum_{\substack{ \cF, \cG \in \cN^\circ(M) \\ \cF \subseteq \cG}}(-1)^{\#\cG - \#\cF} \, \overline{P}_M^\cG(q)  \prod_{F \in \cF} \frac{(q-1)q^{-\rk F}T^{\#F}}{1 - q^{-\rk F}T^{\#F}} ,
\end{align*}
which, by the above, is equal to
\begin{align*}
	& q^{-(\rk M-1)} \sum_{\cG \in \cN^\circ(M)}(-1)^{\#\cG} \,\overline{P}_M^\cG(q) \prod_{F \in \cG} \left( 1 - \frac{(q-1)q^{-\rk F}T^{\#F}}{1 - q^{-\rk F}T^{\#F}}\right)\\
	&=q^{-(\rk M-1)} \sum_{\cF \in \cN^\circ(M)} (-1)^{\#\cF} \,\overline{P}_M^\cF(q) \prod_{F \in \cF} \frac{1- q^{-(\rk F - 1)} T^{\# F}}{1-q^{-\rk F} T^{\# F}}.
\end{align*}
\end{proof}

We now complete the proof of \autoref{functionalequationforreducedmotiviczetafunction}.

\begin{proof}[Proof of \autoref{functionalequationforreducedmotiviczetafunction}]
By \autoref{numericalPoincaredualitymaximalbuildingset} and \autoref{rewritingreducedmotiviczetafunctionintermsofchowpolynomials},
\begin{align*}
	\overline{Z}_M(q^{-1}, T^{-1}) &= q^{\rk M-1} \sum_{\cF \in \cN^\circ(M)} (-1)^{\#\cF} \,\overline{P}_M^\cF(q^{-1}) \prod_{F \in \cF} \frac{1- q^{\rk F - 1} T^{-\# F}}{1-q^{\rk F} T^{-\# F}}\\
	&= \sum_{\cF \in \cN^\circ(M)} (-1)^{\#\cF } \,\overline{P}_M^\cF(q)q^{\#\cF} \prod_{F \in \cF} \frac{q^{-(\rk F -1)}T^{\# F} - 1 }{q(q^{-\rk F}T^{\# F}-1)}\\
	&= \sum_{\cF \in \cN^\circ(M)} (-1)^{\#\cF} \,\overline{P}_M^\cF(q) \prod_{F \in \cF} \frac{1- q^{-(\rk F - 1)} T^{\# F}}{1-q^{-\rk F} T^{\# F}}\\
	&= q^{\rk M-1} \overline{Z}_M(q,T).
\end{align*}
\end{proof}

We will use the remainder of this section to prove \autoref{alternatingsumofcharactersticpolynomialsgivesbackwardscharacteristicpolynomial}.  Note that, if $f(q,T) \in \Z[q^{\pm 1}]\llbracket T \rrbracket_\rat$, then $f(q,T^{-1})$ is an element of $\Z[q^{\pm 1}]\llbracket T \rrbracket_\rat$ as well.  This is because
\[
	(q-1) \frac{q^b (T^{-1})^a}{1-q^b(T^{-1})^a} = -(q-1)\left( 1+ \frac{q^{-b}T^a}{1-q^{-b}T^a} \right) \in \Z[q^{\pm 1}]\llbracket T \rrbracket_\rat.
\]
Let
\[
	\lim_{T \to \infty}: \Z[q^{\pm 1}]\llbracket T \rrbracket_{\rat} \to \Z[q^{\pm 1}]
\]
denote the $\Z[q^{\pm 1}]$-algebra map obtained by composing the involution
\[
	\Z[q^{\pm 1}]\llbracket T \rrbracket_{\rat} \to \Z[q^{\pm 1}]\llbracket T \rrbracket_{\rat}: f(q,T) \mapsto f(q, T^{-1})
\]
with the map
\[
	\Z[q^{\pm 1}]\llbracket T \rrbracket_{\rat} \to \Z[q^{\pm 1}]: f(q,T) \mapsto f(q,0).
\]
We will prove \autoref{alternatingsumofcharactersticpolynomialsgivesbackwardscharacteristicpolynomial} by applying $\lim_{T \to \infty}$ to both sides of the functional equation in \autoref{functionalequationforreducedmotiviczetafunction}.

\begin{proof}[Proof of \autoref{alternatingsumofcharactersticpolynomialsgivesbackwardscharacteristicpolynomial}]
We see that for any $a \in \Z_{>0}$ and $b \in \Z$,
\[
	\lim_{T \to \infty} \left( (q-1) \frac{q^b T^a}{1-q^bT^a} \right) = -(q-1).
\]
Therefore by \autoref{rationalformulaformotiviczetafunctionmaximalbuildingset},
\[
	\lim_{T \to \infty} (q^{\rk M-1} \overline{Z}_M(q,T)) = \sum_{\cF \in \cN^\circ(M)} (-1)^{\#\cF} \overline{\chi}_{M_{\cF}}(q).
\]
Because $T \mapsto T^{-1}$ is an involution, for any $f(q,T) \in \Z[q^{\pm 1}]\llbracket T \rrbracket_{\rat}$,
\[
	\lim_{T \to \infty} f(q,T^{-1}) = f(q,0),
\]
and by definition
\[
	\overline{Z}_M(q, 0) = \overline{\chi}_{M_{\textbf{0}}}(q) q^{-(\rk M-1)-\overline{\wt}_M(\textbf{0})} = q^{-(\rk M-1)} \overline{\chi}_M(q).
\]
Thus
\[
	\lim_{T \to \infty} \overline{Z}_M(q^{-1},T^{-1}) = \overline{Z}_M(q^{-1}, 0) = q^{\rk M-1}\overline{\chi}_M(q^{-1}).
\]
The desired result now follows from \autoref{functionalequationforreducedmotiviczetafunction}.
\end{proof}

\section{A recurrence relation for the local motivic zeta function}

In this section we prove \autoref{recurrencerelationmotiviczetafunctionat0}, which gives a recurrence relation satisfied by the local motivic zeta function of a matroid. We note that, aesthetically, \autoref{recurrencerelationmotiviczetafunctionat0} and its proof are quite similar to \autoref{recursionforpoincarepolynomials} and its proof above.

\begin{proof}[Proof of \autoref{recurrencerelationmotiviczetafunctionat0}]
Let $M$ be a matroid with lattice of flats $\cL$ and ground set $E$. If $M$ has a loop, then both sides of the equation in \autoref{recurrencerelationmotiviczetafunctionat0} are equal to 0, so we will assume that $M$ is loopless. If $\cF \in \cN^*(M)$, then \autoref{initialmatroiddecompositionmatroidminorsmaximalbuildingset} implies
\[
	M_\cF = \bigoplus_{F \in \cF} M |F / z_\cF(F) = (M/ z_\cF(E)) \oplus (M | z_\cF(E))_{\cF \setminus \{E\}},
\]
so
\[
	\chi_{M_\cF}(q) = \chi_{M/z_\cF(E)}(q) \cdot \chi_{(M | z_\cF(E))_{\cF \setminus \{E\}}}(q),
\]
where by slight abuse of notation, we let $\cF \setminus \{E\}$ denote the corresponding element of $\cN^*(M | z_\cF(E))$. Thus by \autoref{rationalformulaformotiviczetafunctionmaximalbuildingset},
\begin{align*}
	&q^{\rk M} Z_{M}^0(q,T) = \sum_{\cF \in \cN^*(M)} \chi_{M_\cF}(q) \prod_{F \in \cF} \frac{q^{-\rk F} T^{\# F}}{1-q^{-\rk F}T^{\# F}} \\
	&= \chi_{M}(q) \frac{q^{-\rk E} T^{\# E}}{1 - q^{-\rk E}T^{\# E}} + \sum_{F \in \widehat\cL} \,\sum_{ \substack{\cF \in \cN^*(M) \\ z_\cF(E) = F}} \chi_{M_\cF}(q) \prod_{G \in \cF} \frac{q^{-\rk G} T^{\# G}}{1-q^{-\rk G}T^{\# G}}\\
	&= \chi_M(q)\frac{q^{-\rk M} T^{\# E}}{1 - q^{-\rk M}T^{\# E}}\\
	&\qquad+ \sum_{F \in \widehat\cL} \chi_{M/F}(q) \frac{q^{-\rk M} T^{\# E}}{1 - q^{-\rk M}T^{\# E}} \sum_{\cF \in \cN^*(M | F)} \chi_{(M | F)_{\cF}}(q) \prod_{G \in \cF} \frac{q^{-\rk G} T^{\# G}}{1-q^{-\rk G}T^{\# G}}\\
	&= (q-1)\frac{q^{-\rk M} T^{\# E}}{1 - q^{-\rk M}T^{\# E}} \left( \barchi_M(q) + \sum_{F \in \widehat\cL} \barchi_{M/F}(q) q^{\rk(M | F)} Z_{M|F}^0 (q, T) \right).
\end{align*}
\end{proof}

\section{The topological zeta function of a matroid}
\label{Sec:Top}

By \cite[Chapter 7 Proposition 3.1.4]{CNS}, there exists a ring morphism
\[
	\mu_{\topo} \colon \Z[q^{\pm 1}]\llbracket T \rrbracket_\rat \to \Q(s),
\]
such that
\[
	\mu_{\topo}(q) = 1
\]
and
\[
	\mu_{\topo} \left( (q-1) \frac{q^b T^a}{1-q^bT^a} \right) = \frac{1}{as - b}
\]
for all $(a,b) \in (\Z_{>0}, \Z)$.  Recall that the topological zeta function of a matroid $M$ is defined to be $Z_M^\topo(s) = \mu_\topo(Z_M(q,T)) \in \Q(s)$. In this section, we will prove \autoref{taylorcoefficientsoftopologicalzetafunction} and \autoref{ourtopologicalspecializestovdv}.

\subsection{Taylor series of the topological zeta function}

We begin by showing that the topological zeta function of a matroid is a specialization of the local motivic zeta function of that matroid.

\begin{proposition}
\label{topologicalzetafunctionisspecializationoflocalmotiviczetafunction}
Let $M$ be a matroid. Then
\[
	Z_M^\topo(s) = \mu_\topo(Z_M^0(q,T)).
\]
\end{proposition}

\begin{proof}
By \autoref{rationalformulaformotiviczetafunctionmaximalbuildingset},
\[
	Z_M^\topo(s) = \mu_\topo(Z_M(q,T)) = \sum_{\cF \in \cN(M)} \frac{\chi_{M_\cF}(q)}{(q-1)^{\#\cF}} \Big|_{q=1} \prod_{F \in \cF} \frac{1}{(\#F)s + \rk F}.
\]
Let $E$ be the ground set of $M$. For any $\cF \in \cN^\circ(M)$,
\[
	\frac{\chi_{M_\cF}(q)}{(q-1)^{\#\cF}} \Big|_{q=1} = \left((q-1)\frac{\chi_{M_{\cF \cup \{E\}}}(q)}{(q-1)^{\#(\cF \cup \{E\})}}\right) \Big|_{q=1} = 0.
\]
Therefore
\[
	Z_M^\topo(s) = \sum_{\cF \in \cN^*(M)} \frac{\chi_{M_\cF}(q)}{(q-1)^{\#\cF}} \Big|_{q=1} \prod_{F \in \cF} \frac{1}{(\#F)s + \rk F} = \mu_\topo(Z_M^0(q,T)).
\]
\end{proof}

\autoref{recurrencerelationmotiviczetafunctionat0} and \autoref{topologicalzetafunctionisspecializationoflocalmotiviczetafunction} immediately imply the next corollary, which gives a recurrence relation satisfied by the topological zeta function of a matroid.

\begin{corollary}
\label{recurrencerelationtopologicalzetafunction}
Let $M$ be a matroid with ground set $E$ and lattice of flats $\cL$. Then
\[
	Z_M^\topo(s) = \frac{1}{(\# E)s + \rk M} \left( \barchi_M(1) + \sum_{F \in \widehat\cL} \barchi_{M/F}(1) Z_{M|F}^\topo(s) \right).
\]
\end{corollary}

We may now complete the proof of \autoref{taylorcoefficientsoftopologicalzetafunction}.

\begin{proof}[Proof of \autoref{taylorcoefficientsoftopologicalzetafunction}]
Let $M$ be a loopless matroid with ground set $E$. We begin by proving that $Z_M^\topo(0) = 1$ by induction on the rank of $M$. \autoref{recurrencerelationtopologicalzetafunction} implies
\begin{align*}
	Z_M^\topo(0) &= \frac{1}{\rk M} \left( \barchi_M(1) + \sum_{F \in \widehat\cL} \barchi_{M/F}(1) Z_{M|F}^\topo(0) \right)\\
	&= \frac{1}{\rk M} \left( \barchi_M(1) + \sum_{F \in \widehat\cL} \barchi_{M/F}(1) \right)\\
	&= 1,
\end{align*}
where the last equality follows from \autoref{redcharsub}. Next we prove that
\[
	\left(\frac{d}{ds} Z^{\topo}_M(s)\right) \Big|_{s = 0} = -\# E.
\]
By \autoref{recurrencerelationtopologicalzetafunction},
\[
	\frac{d}{ds} Z_M^\topo(s) = \frac{1}{(\# E)s + \rk M} \left( (-\# E) Z_M^\topo(s) + \sum_{F \in \widehat\cL} \barchi_{M/F}(1)\frac{d}{ds}Z_{M|F}^{\topo}(s) \right).
\]
Thus by induction on the rank of $M$, we have
\begin{align*}
	\left(\frac{d}{ds} Z^{\topo}_M(s)\right) \Big|_{s = 0} &= \frac{1}{\rk M} \left( -\# E - \sum_{F \in \widehat\cL} (\# F)\barchi_{M/F}(1) \right)\\
	&= \frac{-\# E - (\#E)(\rk M - 1)}{\rk M}\\
	&= -\# E,
\end{align*}
where the second equality follows from \autoref{atomsum}.
\end{proof}

\subsection{The topological zeta function of van der Veer}

In this section we prove \autoref{ourtopologicalspecializestovdv}, which states that in the case of simple matroids, our definition of the topological zeta function of a matroid is equivalent to the definition introduced by van der Veer in \cite{vanderVeer}. To do this, it will be useful to define the following polynomials.

\begin{definition}
\label{def:EP}
Let $M$ be a matroid. For each $\cF \in \cN(M)$, we define the \textbf{Euler-Poincar\'{e} polynomial of $M$ at $\cF$} to be
\[
	P_M^\cF(q) = \sum_{\substack{\cG \in \cN(M) \\ \cG \supseteq \cF}} \frac{\chi_{M_\cG}(q)}{(q-1)^{\# \cG}} \in \Z[q].
\]
We define the \textbf{Euler-Poincar\'{e} polynomial} of $M$ to be
\[
	P_M(q) = \sum_{\cF \in \cN^*(M)} \frac{\chi_{M_\cF}(q)}{(q-1)^{\# \cF}} \in \Z[q].
\]
In other words, $P_M(q)$ is the Euler-Poincar\'{e} polynomial of $M$ at the flag consisting of only the ground set of $M$.
\end{definition}

\begin{remark}
\label{poincareeulerpoincarerelationshipremark}
If $\cF \in \cN^\circ(M)$, then the definitions immediately imply that $P_M^{\cF}(q) = q\overline{P}_M^\cF(q)$ and $P_M^{\cF \cup \{E\}}(q) = \overline{P}_M^{\cF}(q)$, where $E$ is the ground set of $M$. See \autoref{realizablegeometryeulerpoincarepolynomial} below for a geometric interpretation of this fact.

Although this implies that the Poincar\'{e} polynomials and the Euler-Poincar\'{e} polynomials give equivalent information, they evoke different geometric intuitions.
\end{remark}

\begin{remark}
\label{realizablegeometryeulerpoincarepolynomial}
Let $\cA$ be a hyperplane arrangement realizing $M$, let $Y$ be the wonderful model of $\cA$ with respect to the maximal building set, and let $D_{\cF}$ be the closure of the stratum in $Y$ corresponding to $\cF$.  By \autoref{Prop:Poincare}, $P_M^\cF(\bL) = [D_{\cF}] \in K_0(\Var_\C)$, so $P_M^{\cF}(q^2)$ is the Euler-Poincar\'{e} polynomial of $D_\cF$.

Furthermore, suppose that $\cF \in \cN^\circ(M)$, and let $D_{\cF \cup \{E\}}$ be the closure of the stratum in $Y$ corresponding to $\cF \cup \{E\}$, where $E$ is the ground set of $M$. Let $\overline{Y}$ be the wonderful model, with respect to the maximal building set, of the projectivization of $\cA$. Let $\overline{D}_{\cF}$ be the closure of the stratum in $\overline{Y}$ corresponding to $\cF$. By \cite[Theorems 4.1 and 4.2]{DeConciniProcesi}, $D_{\cF}$ is isomorphic to a line bundle on $\overline{D}_\cF$ and $D_{\cF \cup \{E\}}$ is isomorphic to $\overline{D}_\cF$. Thus $P_M^{\cF}(\bL) = \bL\overline{P}_M^\cF(\bL)$ and $P_M^{\cF \cup \{E\}}(\bL) = \overline{P}_M^{\cF}(\bL)$ by \autoref{realizablegeometrypoincarepolynomial}.
\end{remark}

We will now set notation for certain polynomials whose evaluations at 1 can be used to express van der Veer's topological zeta function. For any matroid $M$ and flag $\cF \in \cN^*(M)$, let $H_M^{\cF}(q)$ denote the polynomial
\[
	H_M^\cF(q) = \sum_{\substack{ \cG \in \cN(M) \\ \cG \cup \cF \in \cN(M)}} \prod_{G \in \cG} \left( [\rk G - \rk z_{\cG \cup \cF}(G)]_q - 1\right) \in \Z[q].
\]
For any flag $\cF \in \cN^\circ(M)$, let $H_M^{\cF}(q)$ denote the polynomial
\[
	H_M^\cF(q) = q H_M^{\cF \cup \{E\}}(q),
\]
where $E$ is the ground set of $M$.  These polynomials are generalizations of the polynomials $H_M(q)$ defined in \autoref{proofofpoincarecohomologysection}.  Specifically, we have
\[
	H_M(q) = H_M^{\{E\}}(q).
\]
If $M$ is a simple matroid with lattice of flats $\cL$ and $Z_\cL(s)$ is the topological zeta function of $\cL$ in the sense of \cite[Definition 1]{vanderVeer}, then \cite[Proposition 1]{vanderVeer} implies
\[
	Z_{\cL}(s) = \sum_{\cF \in \cN(M)} \, \sum_{\substack{ \cG \in \cN(M) \\ \cG \supseteq \cF}} (-1)^{\#\cG - \#\cF} H_M^\cG(1) \prod_{F \in \cF} \frac{1}{(\# F)s + \rk F}.
\]
At the same time, \autoref{rationalformulaformotiviczetafunctionmaximalbuildingset} and M\"{o}bius inversion imply
\begin{align*}
	Z_M^\topo(s) &= \sum_{\cF \in \cN(M)} \frac{\chi_{M_\cF}(q)}{(q-1)^{\#\cF}} \Big|_{q=1} \prod_{F \in \cF} \frac{1}{(\#F)s + \rk F}\\
	&= \sum_{\cF \in \cN(M)} \,\sum_{\substack{\cG \in \cN(M)\\ \cG \supseteq \cF}} (-1)^{\# \cG - \# \cF}P_M^\cG(1) \prod_{F \in \cF} \frac{1}{(\#F)s + \rk F}.
\end{align*}
\autoref{ourtopologicalspecializestovdv} therefore follows from the next proposition.

\begin{proposition}
\label{PequalsHmaximalbuildingset}
Let $M$ be a loopless matroid. Then for any $\cF \in \cN(M)$,
\[
	P_M^\cF(q) = H_M^\cF(q).
\]
\end{proposition}

\begin{proof}
Let $E$ be the ground set of $M$. We first prove the special case $\cF = \{E\}$. By \autoref{poincarepolynomialrelatedtomonomialbasis} and \autoref{poincareeulerpoincarerelationshipremark},
\[
	P_M^{\{E\}}(q) = \overline{P}_M(q) = H_M(q) = H_M^{\{E\}}(q).
\]
Now we will prove the case where $\cF \in \cN^*(M)$. In the proof of \autoref{numericalPoincaredualitymaximalbuildingset}, we showed in \eqref{poincareatflagproductofpoincare} that
\[
	P_M^\cF(q) = \overline{P}_M^{\cF \setminus \{E\}}(q) = \prod_{F \in \cF} \overline{P}_{M | F / z_\cF(F)}(q) = \prod_{F \in \cF} P_{M | F / z_{\cF}(F)} (q).
\]
Also
\begin{align*}
	H_M^\cF(q) &= \sum_{\substack{ \cG \in \cN(M) \\ \cG \cup \cF \in \cN(M)}} \prod_{G \in \cG} \left( [\rk G - \rk z_{\cG \cup \cF}(G)]_q - 1\right)\\
	&= \prod_{F \in \cF} \, \sum_{\cG \in \cN(M |F / z_{\cF}(F))} \, \prod_{G \in \cG} \left( [\rk G - \rk_{\cG \cup \{F\}}(G)]_q - 1 \right)\\
	&= \prod_{F \in \cF} H_{M | F / z_{\cF}(F)}(q).
\end{align*}
Therefore by the case where the flag consists only of the ground set, we have
\[
	P_M^\cF(q) = \prod_{F \in \cF} P_{M | F / z_{\cF}(F)} (q) = \prod_{F \in \cF} H_{M | F / z_{\cF}(F)}(q) = H_M^\cF(q).
\]
Finally when $\cF \in \cN^\circ(M)$, \autoref{poincareeulerpoincarerelationshipremark} implies
\[
	P_M^\cF(q) = q P_M^{\cF \cup \{E\}}(q) = q H_M^{\cF \cup \{E\}}(q) = H_M^\cF(q),
\]
and we have completed the proof in all cases.
\end{proof}

\section{Examples}
\label{sectionwithexamples}

The motivic zeta function of a matroid $M$ determines the characteristic polynomial of $M$, because
\[
	\chi_M(q) = q^{\rk M} Z_M(q,0) .
\]
In this section, we show by example that the motivic zeta function is in fact a strictly finer matroid invariant than the characteristic polynomial. We also exhibit a pair of non-isomorphic matroids with the same motivic zeta function.

\begin{example}
Consider the rank $4$ simple matroids $M_1$ and $M_2$ in \autoref{Fig:SameTutte}.

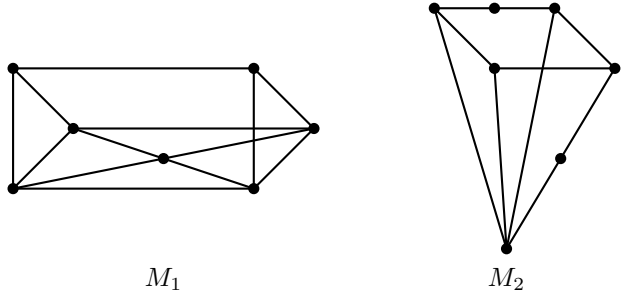
\begin{figure}[h]
\begin{tikzpicture}[thick, scale=0.8]

\begin{scope}[grow=right, baseline]

\draw [color=black, fill=black] (0,0) circle (0.75mm);
\draw [color=black, fill=black] (0,2) circle (0.75mm);
\draw [color=black, fill=black] (4,0) circle (0.75mm);
\draw [color=black, fill=black] (4,2) circle (0.75mm);
\draw [color=black, fill=black] (1,1) circle (0.75mm);
\draw [color=black, fill=black] (5,1) circle (0.75mm);
\draw [color=black, fill=black] (2.5,0.5) circle (0.75mm);

\draw (0,0)--(4,0);
\draw (0,0)--(0,2);
\draw (0,2)--(4,2);
\draw (4,0)--(4,2);
\draw (1,1)--(5,1);
\draw (0,0)--(1,1);
\draw (0,2)--(1,1);
\draw (4,0)--(5,1);
\draw (4,2)--(5,1);
\draw (1,1)--(4,0);
\draw (0,0)--(5,1);

\node at (2.5,-1.5) [] (M1) {$M_1$};

\draw [color=black, fill=black] (7,3) circle (0.75mm);
\draw [color=black, fill=black] (8,3) circle (0.75mm);
\draw [color=black, fill=black] (9,3) circle (0.75mm);
\draw [color=black, fill=black] (8,2) circle (0.75mm);
\draw [color=black, fill=black] (10,2) circle (0.75mm);
\draw [color=black, fill=black] (8.2,-1) circle (0.75mm);
\draw [color=black, fill=black] (9.1,0.5) circle (0.75mm);

\draw (7,3)--(9,3);
\draw (7,3)--(8,2);
\draw (8,2)--(10,2);
\draw (9,3)--(10,2);
\draw (7,3)--(8.2,-1);
\draw (9,3)--(8.2,-1);
\draw (8,2)--(8.2,-1);
\draw (10,2)--(8.2,-1);

\node at (8.2,-1.5) [] (M2) {$M_2$};

\end{scope}
\end{tikzpicture}

\caption{Two simple matroids with the same characteristic polynomial but different motivic zeta functions}
\label{Fig:SameTutte}
\end{figure}

\noindent These matroids have the same characteristic polynomial,
\[
	\chi_{M_1}(q) = \chi_{M_2}(q) = q^4 - 7q^3 + 19q^2 - 23q + 10.
\]
Indeed, they have the same Tutte polynomial.  However,
\[
	Z^{\topo}_{M_1}(s) = \frac{-120s^6 + 20s^5 + 120s^4 - 129s^3 - 29s^2 + 162s + 72}{(s + 1)^3 (3s + 2) (4s + 3) (5s + 3) (7s + 4)},
\]
and
\[
	Z^{\topo}_{M_2}(s) = \frac{-120s^6 + 22s^5 + 120s^4 - 129s^3 - 29s^2 + 162s + 72}{(s + 1)^3 (3s + 2) (4s + 3) (5s + 3) (7s + 4)}.
\]
Since $Z^{\topo}_{M_1}(s) \neq Z^{\topo}_{M_2}(s)$, it follows that $Z_{M_1}(q,T) \neq Z_{M_2}(q,T)$.

\end{example}

\begin{example}
Consider the rank 3 simple matroids $N_1$ and $N_2$ in \autoref{Fig:SameZeta}.

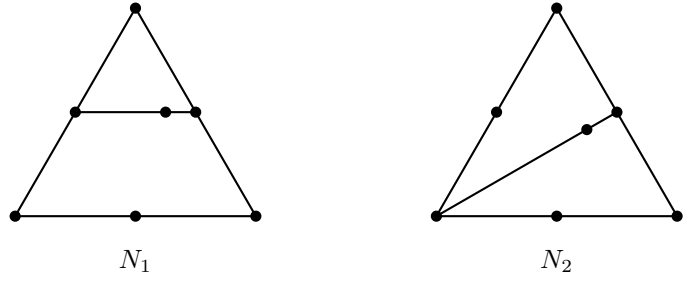
\begin{figure}[h]
\begin{tikzpicture}[thick, scale=0.8]

\begin{scope}[grow=right, baseline]

\draw [color=black, fill=black] (0,0) circle (0.75mm);
\draw [color=black, fill=black] (1,1.73) circle (0.75mm);
\draw [color=black, fill=black] (2,3.46) circle (0.75mm);
\draw [color=black, fill=black] (3,1.73) circle (0.75mm);
\draw [color=black, fill=black] (2,0) circle (0.75mm);
\draw [color=black, fill=black] (4,0) circle (0.75mm);
\draw [color=black, fill=black] (2.5 ,1.73) circle (0.75mm);

\draw (0,0)--(4,0);
\draw (0,0)--(2,3.46);
\draw (4,0)--(2,3.46);
\draw (1,1.73)--(3,1.73);

\node at (2,-.75) [] (N1) {$N_1$};

\draw [color=black, fill=black] (7,0) circle (0.75mm);
\draw [color=black, fill=black] (8,1.73) circle (0.75mm);
\draw [color=black, fill=black] (9,3.46) circle (0.75mm);
\draw [color=black, fill=black] (10,1.73) circle (0.75mm);
\draw [color=black, fill=black] (9,0) circle (0.75mm);
\draw [color=black, fill=black] (11,0) circle (0.75mm);
\draw [color=black, fill=black] (9.5,1.44) circle (0.75mm);

\draw (7,0)--(11,0);
\draw (7,0)--(9,3.46);
\draw (11,0)--(9,3.46);
\draw (7,0)--(10,1.73);

\node at (9,-.75) [] (N2) {$N_2$};

\end{scope}
\end{tikzpicture}

\caption{Two non-isomorphic simple matroids with the same motivic zeta function}
\label{Fig:SameZeta}
\end{figure}

\noindent The matroids $N_1$ and $N_2$ are not isomorphic.  For instance, $N_1$ contains two 3-element flats of rank $2$ with empty intersection, whereas no such pair of flats exists in $N_2$. Nevertheless, $Z_{N_1}(q,T)$ and $Z_{N_2}(q,T)$ are both equal to
\begin{align*}
\frac{1}{(q - T)^2 (q^2 - T^3) (q^3 - T^7)} (q-1)\Big(&3q^4T^3 - q^3T^4 - 11q^2T^5 + q^6 + 5q^5T\\
&+ 3q^4T^2 - 6q^3T^3 + 18q^2T^4 + 6qT^5 - 6q^5\\
&- 18q^4T + 6q^3T^2 - 3q^2T^3 - 5qT^4 - T^5\\
&+ 11q^4 + q^3T - 3q^2T^2\Big).
\end{align*}

\end{example}

\section{Building sets and motivic zeta functions}
\label{arbitrarybuildingsets}

The main results of this paper may be viewed as statements in terms of the maximal building set in the lattice of flats of a matroid.  In this section, we indicate how our results may be extended to the more general framework of arbitrary building sets.  For an example of the utility of such generalizations, consider \autoref{definitionmotiviczetafunctions}.  One might reasonably ask why we choose to define the zeta functions in this way, as opposed to directly defining them as the rational functions appearing in \autoref{rationalformulaformotiviczetafunctionmaximalbuildingset}.  An answer is given by \autoref{rationalformulaformotiviczetafunction}: the definition of a zeta function as a lattice sum is more ``intrinsic,'' while the expression of it as a rational function depends on a choice of building set.

\subsection{Building sets}

Before giving the generalizations of our results, we first survey the basic definitions and properties of building sets.  Building sets and nested sets for (semi)lattices were introduced by Feichtner and Kozlov \cite{FeichtnerKozlov} to generalize notions used in De Concini and Procesi's work on wonderful compactifications \cite{DeConciniProcesi}.

\begin{definition}
Let $M$ be a matroid, and let $\cL$ be its lattice of flats. A subset $\sG$ of $\cL_{> \cl(\emptyset)}$ is a \textbf{building set} if for any $X \in \cL_{> \cl(\emptyset)}$ the set $\max \sG_{\leq X} = \{G_1, \ldots, G_k\}$ satisfies the following: there is an isomorphism of partially ordered sets
\[
	\varphi_X \colon \prod_{i=1}^k \ [\cl(\emptyset), G_i] \ \xrightarrow{\sim} \ [\cl(\emptyset), X]
\]
such that $\varphi_X(\cl(\emptyset), \ldots, G_i, \ldots, \cl(\emptyset)) = G_i$ for $i = 1, \ldots, k$.  The set $\fact_{\sG}(X) = \max \sG_{\leq X}$ is called the set of \textbf{factors} of $X$ in $\sG$.
\end{definition}

We note that there is always a maximal building set $\sG_{\max} = \cL_{> \cl(\emptyset)}$.  There is also a minimal building set $\sG_{\min}$, which consists of all $G \in \cL_{> \cl(\emptyset)}$ such that the interval $[\cl(\emptyset), G]$ cannot be decomposed as a product of smaller intervals.  That is, a flat $G$ is in $\sG_{\min}$ if and only if the restriction $M|G$ is a connected matroid of positive rank.

\begin{definition}
Let $\sG$ be a building set in $\cL$. A subset $\cS \subseteq \sG$ is \textbf{nested} if, for any set of pairwise incomparable elements $G_1, \ldots, G_t \in \cS$ with $t \geq 2$, the join $G_1 \vee \cdots \vee G_t$ does not belong to $\sG$.  The nested sets in $\sG$ form an abstract simplicial complex, the \textbf{nested set complex}, which we denote by $\cN(M, \sG)$.
\end{definition}

\noindent We set
\[
	\cN^*(M, \sG) = \{ \cS \in \cN(M, \sG) \mid \fact_{\sG}(E) \subseteq \cS \}
\]
and
\[
	\cN^{\circ}(M, \sG) = \cN(M, \sG) \setminus \cN^*(M, \sG).
\]
If $\cS \in \cN(M, \sG)$ is a nested set and $F \in \cS$, then we let $z_{\cS}(F) = \bigvee \cS_{<F}$.

\begin{remark}
A subset $\cF$ of $\cL_{> \cl(\emptyset)}$ is nested with respect to the maximal building set if and only if $\cF$ is a flag of flats. That is,
\[
	\cN(M, \sG_{\max}) = \cN(M), \quad \cN^*(M, \sG_{\max}) = \cN^*(M), \quad \text{and} \quad \cN^{\circ}(M, \sG_{\max}) = \cN^{\circ}(M).
\]
\end{remark}

As in \autoref{BergmanFan}, we have fans
\[
	\Sigma(M, \sG) = \{ \sigma_{\cS} \mid \cS \in \cN(M, \sG) \}
\]
and
\[
	\Sigma^{\circ}(M, \sG) = \{ \sigma_{\cS} \mid \cS \in \cN^{\circ}(M, \sG) \}.
\]
When $M$ is loopless, the support of $\Sigma(M, \sG)$ (resp. $\Sigma^{\circ}(M, \sG)$) is $\Trop(M) \cap \R_{\geq 0}^E$ (resp. $\Trop(M) \cap \partial \R_{\geq 0}^E$) \cite[Theorem 4.1]{FeichtnerSturmfels}.  The fans $\Sigma(M, \sG)$ and $\Sigma^{\circ}(M, \sG)$ are also unimodular by \cite[Proposition~2]{FeichtnerYuzvinsky}.

For each $\cS \in \cN(M, \sG)$, the function $w \mapsto M_w$ is constant on $\relint(\sigma_{\cS})$ \cite[Theorem~4.1]{FeichtnerSturmfels}.  We let $M_{\cS}$ denote the matroid $M_w$ for any $w \in \relint(\sigma_{\cS})$.  It is clear that $M_{\cS} = M_{\cS \cup \fact_{\sG}}(E)$, where $E$ is the ground set of $M$. Feichtner and Sturmfels give the following direct sum decomposition of $M_{\cS}$, generalizing \autoref{initialmatroiddecompositionmatroidminorsmaximalbuildingset}.

\begin{proposition}\cite[Theorem~4.4]{FeichtnerSturmfels}
\label{FSinitialmatroiddecomposition}
Let $M$ be a loopless matroid, let $\sG$ be a building set in $\cL_M$, and let $\cS \in \cN^*(M, \sG)$. Then
\[
	M_{\cS} = \bigoplus_{F \in \cS} M|F/z_{\cS}(F).
\]
\end{proposition}

\noindent We note that \autoref{FSinitialmatroiddecomposition} implies that $\chi_{M_{\cS}}(q)$ is divisible by $(q-1)^{\# \cS}$.

\subsection{Building sets and rationality of the motivic zeta functions}

We are now prepared to generalize the results of the paper to arbitrary building sets.  We begin with a generalization of \autoref{rationalformulaformotiviczetafunctionmaximalbuildingset}, which shows that each choice of building set yields a rational formula for the motivic zeta functions of a matroid.

\begin{theorem}
\label{rationalformulaformotiviczetafunction}
Let $M$ be a matroid, and let $\sG$ be a building set in its lattice of flats. Then
\[
	Z_M(q,T) = q^{-\rk M} \sum_{\cS \in \cN(M,\sG)} \frac{\chi_{M_{\cS}}(q)}{(q-1)^{\#\cS}} \prod_{F \in \cS } (q-1) \frac{q^{-\rk F}T^{\#F}}{1 - q^{-\rk F}T^{\#F}},
\]
\[
	Z_M^0(q,T) = q^{-\rk M} \sum_{\cS \in \cN^*(M,\sG)} \frac{\chi_{M_{\cS}}(q)}{(q-1)^{\#\cS}} \prod_{F \in \cS } (q-1) \frac{q^{-\rk F}T^{\#F}}{1 - q^{-\rk F}T^{\#F}},
\]
and
\[
	\overline{Z}_M(q,T) = q^{-(\rk M-1)} \sum_{\cS \in \cN^{\circ}(M,\sG)} \frac{\overline{\chi}_{M_{\cS}}(q)}{(q-1)^{\#\cS}} \prod_{F \in \cS} (q-1) \frac{q^{-\rk F}T^{\#F}}{1 - q^{-\rk F}T^{\#F}}.
\]
\end{theorem}

\begin{proof}
The proof is identical to the proof of \autoref{rationalformulaformotiviczetafunctionmaximalbuildingset}, with the fans $\Sigma(M, \sG)$ and $\Sigma^{\circ}(M, \sG)$ replacing $\Sigma(M)$ and $\Sigma^{\circ}(M)$, respectively.
\end{proof}

One consequence of \autoref{rationalformulaformotiviczetafunction} is the following. If $M$ is a matroid and $\sG$ is a building set in its lattice of flats, then any additive group homomorphism out of $\Z[q^{\pm 1}]\llbracket T \rrbracket_\rat$ induces a specialization of $Z_M(q,T)$ (resp. $Z_M^0(q,T)$, $\overline{Z}_M(q,T)$. This specialization is equal to a sum over $\cN(M, \sG)$ (resp. $\cN^*(M, \sG)$, $\cN^{\circ}(M, \sG)$), and the value of this alternating sum is independent of the choice of building set $\sG$.  We note that this is analogous to the situation in algebraic geometry, where the motivic Igusa zeta function of a hypersurface $X$ may be used to prove that certain expressions in the Grothendieck ring, which are written in terms of a chosen log resolution of $X$, are in fact independent of the choice of log resolution.

For example, \autoref{rationalformulaformotiviczetafunction} may be combined with the functional equation of \autoref{functionalequationforreducedmotiviczetafunction} to obtain the following generalization of \autoref{alternatingsumofcharactersticpolynomialsgivesbackwardscharacteristicpolynomial}.

\begin{corollary}
\label{alternatingsumarbitrarybuildingset}
Let $M$ be a matroid, and let $\sG$ be a building set in its lattice of flats. Then
\[
	\sum_{\cS \in \cN^{\circ}(M, \sG)} (-1)^{\# \cS} \barchi_{\cS}(q) = q^{\rk M - 1} \barchi_M(q^{-1}).
\]
\end{corollary}

\subsection{Building sets and {P}oincar\'e polynomials}

The Euler-Poincar\'e polynomials $P_M^{\cF}(q)$ defined for the maximal building set are easily generalized to arbitrary building sets.

\begin{definition}
Let $M$ be a matroid on ground set $E$, and let $\sG$ be a building set in its lattice of flats. For a nested set $\cS \in \cN(\sG, M)$, we define the \textbf{$\sG$-Euler-Poincar\'{e} polynomial of $M$ at $\cS$} to be
\[
	P_{M,\sG}^{\cS}(q) = \sum_{\substack{\cT \in \cN(M, \sG) \\ \cT \supseteq \cS}} \frac{\chi_{M_\cT}(q)}{(q-1)^{\# \cT}} \in \Z[q].
\]
We define the \textbf{$\sG$-Euler-Poincar\'{e} polynomial} of $M$ to be
\[
	P_{M,\sG}(q) = P_{M,\sG}^{\fact_{\sG}(E)}(q) = \sum_{\cS \in \cN^*(M,\sG)} \frac{\chi_{M_\cS}(q)}{(q-1)^{\# \cS}} .
\]
\end{definition}

\begin{remark}
If $M$ is a matroid on ground set $E$ and $\cS \in \cN(M, \sG)$, then it follows from the definition that $P_{M,\sG}^{\cS}(q) = q^{\#(\fact_{\sG}(E) \setminus \cS)} P_{M, \sG}^{\cS \cup \fact_{\sG}(E)}$. Note that, in all cases, $P_{M, \sG}^{\cS}$ is a polynomial of degree $\rk M - \# \cS$.
\end{remark}

\begin{remark}
If $M$ is realized by a hyperplane arrangement $\cA$, then there is a wonderful model $Y_{\sG}$ determined by the building set $\sG$. The model $Y_{\sG}$ is a compactification of the complement of $\cA$, with strata indexed by the nested set complex $\cN(M, \sG)$ \cite{DeConciniProcesi}. If $D_{\cS}$ is the closure of the stratum in $Y_{\sG}$ corresponding to $\cS \in \cN(M, \sG)$, then $P_{M,\sG}^{\cS}(\bL) = [D_{\cS}] \in K_0(\Var_{\C})$.
\end{remark}

As in \autoref{poincarepolynomialsection}, we shall see that the $\sG$-Euler-Poincar\'{e} polynomial $P_{M, \sG}$ yields the Hilbert series of a certain cohomology ring, the following generalization of $D^\bullet(M)$.

\begin{definition}\cite{FeichtnerYuzvinsky}
Let $M$ be a matroid, and let $\sG$ be a building set in its lattice of flats $\cL$. The \textbf{$\sG$-cohomology ring} of $M$ is the graded algebra
\[
	D^\bullet(M, \sG) = \bigoplus_i D^i(M, \sG) =  \Z[ \{x_G\}_{G \in \sG}] / \mathcal{I}_{\sG},
\]
where each $x_G$ has degree 2, and $\mathcal{I}_{\sG}$ is the ideal generated by
\[
	\prod_{i = 1}^k x_{G_i}, \quad \text{for $\{G_1, \dots, G_k\} \notin \cN(M, \sG)$},
\]
and
\[
	\sum_{G \supseteq A} x_G, \quad \text{for $A$ an atom of $\cL$}.
\]
\end{definition}

We now define analogues of the polynomials $H_M^{\cF}(q)$.

\begin{definition}
Let $M$ be a matroid, and let $\sG$ be a building set in its lattice of flats. For $S \in \cN^*(M, \sG)$, let
\[
	H_{M, \sG}^{\cS}(q) = \sum_{\substack{\cT \in \cN(M, \sG) \\ \cT \cup \cS \in \cN(M, \sG)}} \prod_{F \in \cT} \left( [\rk F - \rk z_{\cT \cup \cS}(F)]_q - 1\right) \in \Z[q]
\]
and for $\cS \in \cN(M, \sG)$, set
\[
	H_{M, \sG}^{\cS}(q) = q^{\# (\fact_{\sG}(E) \setminus \cS)} H_{M, \sG}^{\cS \cup \fact_{\sG}(E)}(q).
\]
Define also
\[
	H_{M, \sG}(q) = H_{M, \sG}^{\fact_{\sG}(E)}(q) = \sum_{\cS \in \cN(M, \sG)} \prod_{F \in \cS} \left( [\rk F - \rk z_{\cS}(F)]_q - 1\right) .
\]
\end{definition}

It follows from \cite[Section~3, Corollary~1]{FeichtnerYuzvinsky} that
\[
	H_{M, \sG}(q^2) = \sum_{i \geq 0} \rk_{\Z} D^i(M, \sG) q^i.
\]
Therefore, the following generalization of \autoref{thm:hilbertseries} may be proved by showing that $P_{M, \sG}(q) = H_{M, \sG}(q)$ for every loopless matroid $M$.

\begin{theorem}
\label{thm:hilbertseriesgeneral}
Let $M$ be a loopless matroid, and let $\sG$ be a building set in its lattice of flats. Then
\[
	P_{M, \sG}(q^2) = \sum_{i \geq 0} \rk_{\Z} D^i(M, \sG) q^i.
\]
\end{theorem}

More generally, the polynomials $P_{M, \sG}^{\cS}$ and $H_{M, \sG}^{\cS}$ agree for any nested set $\cS$.

\begin{theorem}
\label{PequalsH}
Let $M$ be a loopless matroid, let $\sG$ be a building set in $\cL_M$, and let $\cS \in \cN(M, \sG)$ be a nested set. Then
\[
	P_{M, \sG}^{\cS}(q) = H_{M, \sG}^{\cS}(q).
\]
\end{theorem}

The proofs of \autoref{thm:hilbertseriesgeneral} and \autoref{PequalsH} are broadly analogous to the proofs of the special cases,  \autoref{thm:hilbertseries} and \autoref{PequalsHmaximalbuildingset}, respectively, given above.  For brevity, we omit these proofs here.  However, we note that in order to generalize the methods used for $\sG_{\max}$ to arbitrary building sets, it is useful to understand how a building set for $M$ gives rise to a building set for any matroid minor of $M$.  The following proposition accomplishes this.

\begin{proposition}
\label{buildingsetrestrictioncontraction}
Let $M$ be a matroid on ground set $E$, let $\sG$ be a building set in its lattice of flats $\cL$, and let $X \in \cL$ be a flat. Then
\[
	\sG|X = \sG_{\leq X} = \{ G \in \sG \mid G \subseteq X \}
\]
is a building set in the lattice of flats of the restriction $M|X$, and
\[
	\sG/X = \{ G \vee X \mid G \in \sG \setminus \sG_{\leq X} \}
\]
is a building set in the lattice of flats of the contraction $M/X$.
\end{proposition}

We remark that \autoref{buildingsetrestrictioncontraction} may be proved using the main results of \cite{FeichtnerKozlov}.

\subsection{Building sets and the topological zeta function}

In \cite{vanderVeer}, the topological zeta function of a finite, ranked, atomic lattice $\cL$ with respect to a building set $\sG$ is defined to be
\[
	Z_{\cL, \sG}(s) = \sum_{\cS \in \cN(M, \sG)} \sum_{\substack{\cT \in \cN(M, \sG) \\ \cT \supseteq \cS}} (-1)^{\# \cT - \# \cS} H_{M, \sG}^{\cT}(1) \prod_{F \in \cS} \frac{1}{(\# F)s + \rk F}.
\]
The main theorem of \cite{vanderVeer} is that this definition is independent of the choice of building set $\sG$. In the case where $\cL$ is the lattice of flats of a simple matroid $M$, the results of this section provide an alternate proof of this independence, as follows.

By definition of $P_{M, \sG}^{\cS}$ and M\"obius inversion in $\cN(M, \sG)$, we have
\[
	\frac{\chi_{M_{\cS}}(q)}{(q-1)^{\# \cS}} = \sum_{\substack{\cT \in \cN(M, \sG) \\ \cT \supseteq \cS}} (-1)^{\# \cT - \# \cS} P_{M, \sG}^{\cT}(q).
\]
Therefore, by \autoref{rationalformulaformotiviczetafunction} and \autoref{PequalsH}, we have
\begin{align*}
	Z_M^{\topo}(s)
	&= \sum_{\cS \in \cN(M,\sG)} \frac{\chi_{M_{\cS}}(q)}{(q-1)^{\#\cS}} \Big|_{q=1} \prod_{F \in \cS } \frac{1}{(\# F)s + \rk F} \\
	&= \sum_{\cS \in \cN(M,\sG)} \sum_{\substack{\cT \in \cN(M, \sG) \\ \cT \supseteq \cS}} (-1)^{\# \cT - \# \cS} P_{M, \sG}^{\cT}(1) \prod_{F \in \cS } \frac{1}{(\# F)s + \rk F} \\
	&= Z_{\cL, \sG}(s).
\end{align*}
That is, the various rational expressions given by van der Veer are equal because each is equal to $\mu_{\topo} (Z_M(q,T))$.

\bibliographystyle{alpha}
\bibliography{MZFM}

\begin{thebibliography}{CLNS18}

\bibitem[AHK18]{AdiprasitoHuhKatz}
Karim Adiprasito, June Huh, and Eric Katz.
\newblock Hodge theory for combinatorial geometries.
\newblock {\em Ann. of Math. (2)}, 188(2):381--452, 2018.

\bibitem[AK06]{ArdilaKlivans}
Federico Ardila and Caroline~J. Klivans.
\newblock The {B}ergman complex of a matroid and phylogenetic trees.
\newblock {\em J. Combin. Theory Ser. B}, 96(1):38--49, 2006.

\bibitem[Bit04]{Bittner}
Franziska Bittner.
\newblock The universal {E}uler characteristic for varieties of characteristic
  zero.
\newblock {\em Compos. Math.}, 140(4):1011--1032, 2004.

\bibitem[CLNS18]{CNS}
Antoine Chambert-Loir, Johannes Nicaise, and Julien Sebag.
\newblock {\em Motivic integration}, volume 325 of {\em Progress in
  Mathematics}.
\newblock Birkh\"{a}user/Springer, New York, 2018.

\bibitem[DCP95]{DeConciniProcesi}
C.~De~Concini and C.~Procesi.
\newblock Wonderful models of subspace arrangements.
\newblock {\em Selecta Math. (N.S.)}, 1(3):459--494, 1995.

\bibitem[DL92]{DenefLoeser1992}
J.~Denef and F.~Loeser.
\newblock Caract\'{e}ristiques d'{E}uler-{P}oincar\'{e}, fonctions z\^{e}ta
  locales et modifications analytiques.
\newblock {\em J. Amer. Math. Soc.}, 5(4):705--720, 1992.

\bibitem[DL98]{DenefLoeser1998}
Jan Denef and Fran\c{c}ois Loeser.
\newblock Motivic {I}gusa zeta functions.
\newblock {\em J. Algebraic Geom.}, 7(3):505--537, 1998.

\bibitem[DL01]{DenefLoeser2001}
Jan Denef and Fran\c{c}ois Loeser.
\newblock Geometry on arc spaces of algebraic varieties.
\newblock In {\em European {C}ongress of {M}athematics, {V}ol. {I}
  ({B}arcelona, 2000)}, volume 201 of {\em Progr. Math.}, pages 327--348.
  Birkh\"auser, Basel, 2001.

\bibitem[DL02]{DenefLoeser2002}
Jan Denef and Fran\c{c}ois Loeser.
\newblock Lefschetz numbers of iterates of the monodromy and truncated arcs.
\newblock {\em Topology}, 41(5):1031--1040, 2002.

\bibitem[EPW16]{EliasProudfootWakefield}
Ben Elias, Nicholas Proudfoot, and Max Wakefield.
\newblock The {K}azhdan-{L}usztig polynomial of a matroid.
\newblock {\em Adv. Math.}, 299:36--70, 2016.

\bibitem[Eur20]{Eur}
Christopher Eur.
\newblock Divisors on matroids and their volumes.
\newblock {\em J. Combin. Theory Ser. A}, 169:105135, 2020.

\bibitem[FK04]{FeichtnerKozlov}
Eva-Maria Feichtner and Dmitry~N. Kozlov.
\newblock Incidence combinatorics of resolutions.
\newblock {\em Selecta Math. (N.S.)}, 10(1):37--60, 2004.

\bibitem[FS05]{FeichtnerSturmfels}
Eva~Maria Feichtner and Bernd Sturmfels.
\newblock Matroid polytopes, nested sets and {B}ergman fans.
\newblock {\em Port. Math. (N.S.)}, 62(4):437--468, 2005.

\bibitem[FY04]{FeichtnerYuzvinsky}
Eva~Maria Feichtner and Sergey Yuzvinsky.
\newblock Chow rings of toric varieties defined by atomic lattices.
\newblock {\em Invent. Math.}, 155(3):515--536, 2004.

\bibitem[HK12]{HelmKatz}
David Helm and Eric Katz.
\newblock Monodromy filtrations and the topology of tropical varieties.
\newblock {\em Canad. J. Math.}, 64(4):845--868, 2012.

\bibitem[Kon95]{Kontsevich}
Maxim Kontsevich.
\newblock String cohomology, December 1995.
\newblock Lecture at {Orsay}.

\bibitem[KU18]{KutlerUsatine}
Max {Kutler} and Jeremy {Usatine}.
\newblock {Motivic zeta functions of hyperplane arrangements}.
\newblock {\em arXiv e-prints}, page arXiv:1810.11552, Oct 2018.

\bibitem[LRS17]{LopezdeMedranoRinconShaw}
Lucia {Lopez de Medrano}, Felipe {Rincon}, and Kristin {Shaw}.
\newblock {Chern-Schwartz-MacPherson cycles of matroids}.
\newblock {\em arXiv e-prints}, page arXiv:1707.07303, Jul 2017.
\newblock To appear in Proc. Lond. Math. Soc.

\bibitem[Nel18]{Nelson}
Peter Nelson.
\newblock Almost all matroids are nonrepresentable.
\newblock {\em Bull. Lond. Math. Soc.}, 50(2):245--248, 2018.

\bibitem[Tev07]{Tevelev}
Jenia Tevelev.
\newblock Compactifications of subvarieties of tori.
\newblock {\em Amer. J. Math.}, 129(4):1087--1104, 2007.

\bibitem[vdV18]{vanderVeer}
Robin \SmallHack Van \SmallHack Der~Veer\Bibkeyhack vdV.
\newblock {Combinatorial analogs of topological zeta functions}.
\newblock {\em ArXiv e-prints}, March 2018.

\end{thebibliography}

\end{document}